\documentclass{amsart}
\usepackage{amsmath, amssymb}
\usepackage[mathscr]{eucal}
\usepackage{amsfonts,amscd}
\usepackage{amsthm,latexsym}
\usepackage{fullpage}
\usepackage{bm}
\usepackage{graphicx}
\usepackage{enumitem}
\usepackage[all]{xy}
\newtheorem{thm}{Theorem}[section]

\newtheorem{prop}[thm]{Proposition}
\renewenvironment{proof}{\par\noindent{\bf Proof.}}{$\square$\par\bigskip}
\newtheorem{lem}[thm]{Lemma}
\newtheorem{remark}[thm]{Remark}
\newtheorem{cor}[thm]{Corollary}

\def\Z{\mathbb Z}

\def\R{\mathbb R}
\def\C{\mathbb C}
\def\F{\mathbb F}
\def\F{\mathcal F}

\def\O{\operatorname{O}}
\def\o{\operatorname{o}}
\def\log{\operatorname{log}}

\def\thet{\operatorname{\theta_{\textit{f}}\,(\textit{p})}}
\def\tht{\operatorname{\frac{\theta_{\textit{f}}\,(\textit{p})}{2\pi}}}

\def\A{\operatorname{\hat{\mathcal{S}}^{\pm}_M}}
\def\Ap{\operatorname{\hat{\mathcal{S}}^{+}_M}}

\def\T{\operatorname{\widehat{{\textit{S}^\pm_\textit{M}}^2}}}
\def\U{\operatorname{\mathcal{U}^{\pm}_M}}

\def\st{\operatorname{\hat{\it{S}}_\textit{M}^{\pm}}}

\def\ST{\operatorname{\textit{S}^{\pm}(\textit{M,\,f\,})(\textit{x})}}

\date\today
\title{Fluctuations in the distribution of Hecke eigenvalues about the Sato-Tate measure}
\author[Neha Prabhu]{Neha Prabhu}
\address{Neha Prabhu, IISER Pune, Dr Homi Bhabha Road, Pashan, Pune - 411008, Maharashtra, India}
\email{neha.prabhu@students.iiserpune.ac.in}
\author[Kaneenika Sinha]{Kaneenika Sinha}
\address{Kaneenika Sinha, IISER Pune, Dr Homi Bhabha Road, Pashan, Pune - 411008, Maharashtra, India}
\email{kaneenika@iiserpune.ac.in}
\thanks{The research of the first author is supported by a PhD scholarship from the National Board for Higher Mathematics, India.}

\begin{document}

\begin{abstract}
We study fluctuations in the distribution of families of $p$-th Fourier coefficients $a_f(p)$ of normalised holomorphic Hecke eigenforms $f$ of weight $k$ with respect to $SL_2(\Z)$ as $k \to \infty$ and primes $p \to \infty.$  These families are known to be equidistributed with respect to the Sato-Tate measure.  We consider a fixed interval $I \subset [-2,2]$  and derive the variance of the number of $a_f(p)$'s lying in $I$ as $p \to \infty$ and $k \to \infty$ (at a suitably fast rate).  The number of $a_f(p)$'s lying in $I$ is shown to asymptotically follow a Gaussian distribution when appropriately normalised.  A similar theorem is obtained for primitive Maass cusp forms.
\end{abstract}

\keywords{Hecke eigenforms, Sato-Tate distribution, distribution of eigenvalues of Hecke operators, Maass forms}
\subjclass[2016]{Primary 11F11, 11F25, Secondary 11F30}

\maketitle

\section{Introduction}
The statistical distribution of eigenvalues of the Hecke operators acting on spaces of modular cusp forms and Maass forms has been well investigated in recent years (\cite{BGHT},\,\cite{Sarnak},\,\cite{Serre}).  Among the early developments that motivated this study was a famous conjecture, stated independently by M. Sato and J. Tate around 1960. This conjecture predicted a distribution law for the second order terms in the expression for the number of points in a non-CM elliptic curve modulo a prime $p$ as the primes vary. Serre \cite{Serre1} generalised this conjecture in 1968 to the context of modular forms.  The modular version of the Sato-Tate conjecture can be understood as follows:

Let $k$ be a positive even integer and $N$ be a positive integer.  Let  $S(N,k)$ denote the space of modular cusp forms of weight $k$ with respect to $\Gamma_0(N).$  For $n \geq 1,$ let $T_n$ denote the $n$-th Hecke operator acting on $S(N,k).$  We denote the set of all newforms  in $S(N,k)$  by $\mathcal F_{N,k}.$  Any $f(z) \in \mathcal F_{N,k}$ has a Fourier expansion
$$f(z) = \sum_{n=1}^{\infty} {n^{\frac{k-1}{2}}}a_f(n) q^n,$$
where 
$a_f(1) = 1$ and
$$\frac{T_n(f(z))}{n^{\frac{k-1}{2}}} = a_f(n) f(z),\,n \geq 1.$$
Let $p$ be a prime number such that $\gcd\,(p,N) = 1.$  By a theorem of Deligne \cite{Deligne}, the eigenvalues $a_f(p)$ lie in the interval $[-2,2].$  One can study the distribution of the coefficients $a_f(p)$ in different ways:
\begin{enumerate}
\item [(A)] (Sato-Tate family) Let $N$ and $k$ be fixed and let $f(z)$ be a non-CM newform in $\mathcal F_{N,k}.$  We consider the sequence $\{a_f(p)\}$ as $p \to \infty.$
\item [(B)](Vertical Sato-Tate family) For a fixed prime $p,$ we consider the families 
$$\{a_f(p),\,f \in \mathcal F_{N,k}\},\,|\mathcal F_{N,k}| \to \infty.$$  
\item [(C)](Average Sato-Tate family) We consider the families 
$$\{a_f(p),\,p \leq x,\,f \in \mathcal F_{N,k}\},\,|\mathcal F_{N,k}| \to \infty,\,x \to \infty.$$  
\end{enumerate}
Serre's modular version of the Sato-Tate conjecture predicts a distribution law for the sequence defined in (A).  More explicitly, let $I$ be a subinterval of $[-2,2]$ and for a positive real number $x$ and $f \in \mathcal F_{N,k},$ let
$$N_I(f,x): = \#\{p \leq x:\,\gcd\, (p,N) = 1,\,a_f(p) \in I\}.$$
The Sato-Tate conjecture states that for a fixed non-CM newform $f\in \mathcal F_{N,k},$ we have
$$\lim_{x \to \infty} \frac{N_I(f,x)}{\pi(x)} = \int_I\mu_{\infty}(t)dt,$$
where $\pi(x)$ denotes the number of primes less than or equal to $x$ and 
$$\mu_{\infty}(t) := 
\begin{cases}
\frac{1}{\pi}\sqrt{1 - \frac{t^2}{4}} &\text{ if }t \in [-2,2]\\
0 &\text{ otherwise.}
\end{cases}
$$
The measure $\mu_{\infty}(t)$ is referred to as the Sato-Tate or semicircle measure in the literature.  This conjecture has deep and interesting generalisations and has been a central theme in arithmetic geometry over the last few decades.  In 1970, Langlands \cite{Langlands} formulated a general automorphy conjecture which would imply the Sato-Tate conjecture.  This general automorphy conjecture is still open. However, using a very special case of the Langlands functoriality conjecture, M. R. Murty and V. K. Murty \cite{Murty-Murty} have shown that the general automorphy conjecture follows.

The Sato-Tate conjecture has now been proved in the highly celebrated work of Barnet-Lamb, Geraghty, Harris and Taylor \cite{BGHT}.  The methods in \cite{BGHT} to address the Sato-Tate conjecture are different from the approach of Langlands: the authors 
prove that the $L$-functions $L_m(s)$ associated to symmetric powers of $l$-adic representations ($l$ coprime to $N$) attached to $f$ are {\it potentially} automorphic.

If these $L$-functions are automorphic, then one can also obtain error terms in the Sato-Tate distribution.  In fact, under the condition that all symmetric power $L$-functions are automorphic and satisfy the Generalized Riemann Hypothesis,  V. K. Murty \cite{VKMurty} showed that for a non-CM newform $f$ of weight 2 and square free level $N,$ we have
$$N_I(f,x) = \pi(x)\int_I\mu_{\infty}(t)dt + \O\left(x^{3/4}\sqrt{\log\,Nx}\right).$$ 
Building on Murty's work, Bucur and Kedlaya \cite{BK} have obtained, under some analytic assumptions on motivic $L$-functions, an extension of the effective Sato-Tate error term for arbitrary motives. 
Recently, Rouse and Thorner \cite{RT} have generalised Murty's explicit result for all squarefree $N$ and even $k \geq 2,$ further improving the error term by a factor of  $\sqrt{\log\,Nx}.$

In 1987, Sarnak \cite{Sarnak} shifted perspectives and considered a vertical variant of the Sato-Tate conjecture in the case of primitive Maass cusp forms.  For a fixed prime $p,$ he obtained a distribution measure for the $p$-th coefficients of Maass Hecke eigenforms averaged over Laplacian eigenvalues.  

In 1997, Serre \cite{Serre} considered a similar vertical question for holomorphic Hecke eigenforms.  For a fixed prime $p,$ let  $|\mathcal F_{N,k}| \to \infty$ such that $k$ is a positive even integer and $N$ is coprime to $p.$  Let $I$ be a subinterval of $[-2,2]$ and 
$$N_I(N,k) := \#\{f \in \mathcal F_{N,k}:\,a_f(p) \in I \}.$$
Serre showed that
\begin{equation}\label{Serre-ed}
\lim_{|\mathcal F_{N,k}| \to \infty} \frac{N_I(N,k)}{|\mathcal F_{N,k}|} = \int_I \mu_p(t)dt,
\end{equation}
 where
 $$\mu_p(t) = 
\begin{cases}
\frac{p+1}{\pi}\frac{(1-t^2/4)^{1/2}}{(p^{1/2}+p^{-1/2})^2-t^2} & \text{ if }t \in [-2,2]\\
0 & \text{otherwise}.
\end{cases}$$
That is,
$$\mu_p(t) = \frac{(p+1)}{(p^{1/2}+p^{-1/2})^2-t^2}\mu_{\infty}(t).$$
The measure $\mu_p(t)$ is referred to as the $p$-adic Plancherel measure in the literature.  This theorem was independently proved by Conrey, Duke and Farmer \cite{CDF} for $N = 1.$ 

Since averaging over eigenforms provides us with an important tool namely, the Eichler-Selberg trace formula, the quantity $N_I(N,k)$ becomes easier to approach.  Error terms in Serre's theorem were obtained by M. R. Murty and K. Sinha \cite{MS}.  They prove that for a positive integer $N,$ a prime number $p$ coprime to $N$ and a subinterval $I$ of $[-2,2],$
\begin{equation*}
N_I(N,k) = |\mathcal F_{N,k}|\int_I \mu_p(t)dt + \O\left(\frac{|\mathcal F_{N,k}|\log p}{\log kN}\right).
\end{equation*}
In this note, we consider the families described in (C),
$$\{a_f(p),\,p \leq x,\,f \in \mathcal F_{N,k}\}$$
as $|\mathcal F_{N,k}| \to \infty$ and $x \to \infty.$  In other words, this is the Sato-Tate family (A) averaged over all newforms in $\mathcal F_{N,k}.$  In fact, in this direction, the following theorem was proved by Conrey, Duke and Farmer \cite{CDF}:  If $x \to \infty$ and $ k = k(x)$ satisfies $\frac{\log k}{x} \to \infty,$ then, for any subinterval $I$ of $[-2,2],$
 $$\lim_{x \to \infty} \dfrac{1}{|\mathcal F_{1,k}|}\sum_{f \in \mathcal F_{1,k}}\frac{N_I(f,x)}{\pi(x)} = \int_I\mu_{\infty}(t)dt.$$
Nagoshi \cite{Nagoshi} obtained the same asymptotic under weaker conditions on the growth of $k,$ namely, $k = k(x)$ satisfies $\frac{\log k}{\log x} \to \infty$ as $x \to \infty.$  An effective version of Nagoshi's theorem was proved by Wang \cite{Wang}.  Under the above mentioned conditions, he proves that
$$\dfrac{1}{|\mathcal F_{1,k}|}\sum_{f \in \mathcal F_{1,k}}\frac{N_I(f,x)}{\pi(x)} = \int_I\mu_{\infty}(t)dt + \O\left(\frac{\log x}{\log k} + \frac{\log x\log \log x}{x}\right).$$
We also note that although Conrey, Duke and Farmer \cite{CDF} and Nagoshi \cite{Nagoshi} state their ``average" Sato-Tate theorems for $N=1,$ one can easily generalise their techniques to $N >1$.  One can show that if $k$ runs over all positive even integers such that $\frac{\log k}{\log x} \to \infty$ as $x \to \infty,$ then
$$\lim_{x \to \infty} \dfrac{1}{|\mathcal F_{N,k}|}\sum_{f \in \mathcal F_{N,k}}\frac{N_I(f,x)}{\pi(x)} = \int_I\mu_{\infty}(t)dt.$$
In this note, for simplicity of computation and exposition, we assume that $N = 1.$  Henceforth, we denote $\mathcal F_{1,k}$ by $\mathcal F_k$ and $|\mathcal F_{1,k}|$ by $s_k.$  
 The ``average" Sato-Tate theorem tells us that for a fixed interval $I,$ the expected value of $N_I(f,x)$ as we vary $f \in \mathcal F_{k},$ 
 $$E[N_I(f,x)] :=\frac{1}{s_k}\sum_{f \in \mathcal F_{k}} N_I(f,x),$$
 is asymptotic to
$$\pi(x)\int_I\mu_{\infty}(t)dt$$
as $x \to \infty$ with $\frac{\log k}{\log x} \to \infty.$  It is therefore natural to ask what we can say about the fluctuations of $N_I(f,x)$ about the expected value.  In this direction, we prove that under appropriate conditions on the growth of $k = k(x),$ $N_I(f,x)$ has variance asymptotic to 
$$\pi(x)\left[\mu_{\infty}(I) - (\mu_{\infty}(I))^2\right],$$
where 
$$\mu_{\infty}(I) :=\int\limits_{I}\mu_\infty(t)dt.$$
Finally, when appropriately normalised, the limiting distribution of the random variable 
$$ \frac{N_I(f,x) - {\pi(x)}\mu_{\infty}(I)}{\sqrt{\pi(x) \left[\mu_{\infty}(I) - \left(\mu_{\infty}(I)\right)^2\right]}}$$  as $x \to \infty$ is Gaussian, provided the weight $k = k(x)$ grows appropriately faster than the range of primes $p \leq x.$  More precisely, we prove the following theorem:
\begin{thm}\label{main-eigenvalues}
Let $I = [a,b]$ be a fixed interval in $[-2,2].$  As defined above, for a positive real number $x$ and $f \in \mathcal F_{k},$ let
$$N_I(f,x) = \#\{p \leq x:\, a_f(p) \in I\}.$$
Suppose that $k=k(x)$ satisfies $\frac{\log k}{\sqrt{x} \log x} \rightarrow \infty$ as $x \rightarrow \infty$.
Then for any bounded, continuous, real-valued function $g$ on $\R,$ we have
\begin{equation*}
\lim_{x \to \infty} \dfrac{1}{s_k}\sum\limits_{f\in \mathcal{F}_{k}} g \left( \frac{N_I(f,x) - {\pi(x)}\mu_{\infty}(I)}{\sqrt{\pi(x) \left[\mu_{\infty}(I) - \left(\mu_{\infty}(I)\right)^2\right]}}\right) = \dfrac{1}{\sqrt{2 \pi}} \int\limits_{-\infty}^{\infty}g(t)e^{-\frac{t^2}{2}}dt.   
\end{equation*}
In other words, for any real numbers $A < B,$  
$$\lim_{x \to \infty} \text{ Prob }_{\mathcal F_{k}} \left(A < \frac{N_I(f,x) - {\pi(x)}\mu_{\infty}(I)}{\sqrt{\pi(x) \left[\mu_{\infty}(I) - \left(\mu_{\infty}(I)\right)^2\right]}} < B \right) = \frac{1}{\sqrt{2\pi}}\int_A^B e^{-t^2/2}dt.$$
\end{thm}
\subsection{Harmonic averaging}\label{Harmonic-averaging-theorem}
We can also consider a weighted variant of the statistical questions posed in this article.  Instead of uniformly averaging over cusp forms in $\mathcal F_k,$ we consider the case of harmonic averaging.  That is, for $f \in \mathcal F_k,$ we denote
$$\omega_f := \frac{\Gamma(k-1)}{(4\pi)^{k-1}\langle f,f \rangle},$$
where $\langle f,g \rangle$ denotes the Petersson inner product of $f,\,g \in S_k.$  We define
$$h_k := \sum_{f \in \mathcal F_k}\omega_f.$$
For a function $\phi:\,S_k \to \C,$
we denote its harmonic average as follows:
$$\langle \phi(f) \rangle_{h_k} := \frac{1}{h_k}\sum_{f \in \mathcal F_k}\omega_f \phi(f).$$  We can prove the following analogue of Theorem \ref{main-eigenvalues} with harmonic weights attached to the quantities in consideration.
\begin{thm}\label{main-eigenvalues-harmonic}
Let $I = [a,b]$ be a fixed interval in $[-2,2].$  Suppose that $k=k(x)$ satisfies $\frac{\log k}{ \sqrt{x}\log x} \rightarrow \infty$ as $x \rightarrow \infty$.
Then for any bounded, continuous, real-valued function $g$ on $\R,$ we have
\begin{equation*}
\lim_{x \to \infty} \dfrac{1}{h_k}\sum\limits_{f\in \mathcal{F}_{k}} \omega_f \, g \left( \frac{ N_I(f,x) - {\pi(x)}\mu_{\infty}(I)}{\sqrt{\pi(x) \left[\mu_{\infty}(I) - \left(\mu_{\infty}(I)\right)^2\right]}}\right) = \dfrac{1}{\sqrt{2 \pi}} \int\limits_{-\infty}^{\infty}g(t)e^{-\frac{t^2}{2}}dt.   
\end{equation*}
\end{thm}

\subsection{Maass cusp forms}\label{MaassForms}
The case of primitive Maass cusp forms admits a similar analysis to what we present in this article for holomorphic modular cusp forms.  We therefore make some observations in this case.  

Let $\mathcal C(\Gamma \backslash \mathfrak H)$ denote the space of Maass cusp forms with respect to $\Gamma = SL_2(\Z).$  Let $\{f_j:\,j \geq 0\}$ denote an orthonormal basis for $\mathcal C(\Gamma \backslash \mathfrak H),$ which consists of the simultaneous eigenforms of the non-Euclidean Laplacian operator $\Delta$ and Hecke operators $T_n,\,n \geq 1.$  Here, let $f_0$ denote the constant function.  For an eigenform $f_j,$ we have
$$\Delta f_j = \left(\frac{1}{4} + t_j^2\right)f_j,\quad T_n f_j = a_j(n)f_j.$$
For $z = x + iy \in \mathfrak H,$ each $f_j$ has the Fourier expansion
$$f_j(z) = \sqrt y \varrho_j(1) \sum_{n = 1}^{\infty} a_j(n) K_{it_j}(2\pi|n|y)e(nx),$$
where $a_j(n) \in \R,$ $\varrho_j(1) \neq 0$ and $K_{\nu}$ is the $K$-Bessel function of order $\nu.$  We order the $f_j$'s so that $0 < t_1 \leq t_2 \leq t_3 \leq \dots.$  It is well known, by a result of Weyl, that
\begin{equation}\label{Weyl} 
r(T):=\#\{j:\,0<t_j \leq T\} = \frac{1}{12}T^2 + \O(T\log T).
\end{equation}
The Ramanujan-Petersson conjecture, which is still open, is the assertion that for all primes $p,$
$$|a_j(p)| \leq 2.$$

For an interval $I = [a,b] \subset \R$ and for $1 \leq j \leq r(T),$ let us define
$$N_I(j,x) = \#\{p \leq x:\,a_j(p) \in I\}.$$
We have the following analogue of Theorem \ref{main-eigenvalues} for Maass cusp forms.
\begin{thm}\label{main-eigenvalues-Maass}
Suppose that $T=T(x)$ satisfies $ \frac{\log T}{\sqrt{x}\log x} \rightarrow \infty$ as $x \rightarrow \infty$.  Let $I = [a,b] $ be a fixed interval in $ \R.$  Then for any bounded, continuous, real-valued function $g$ on $\R,$ we have
\begin{equation*}
\lim_{x \to \infty} \dfrac{1}{r(T)}\sum\limits_{j=1}^{r(T)} g \left( \frac{N_I(j,x) - {\pi(x)}\mu_{\infty}(I)}{\sqrt{\pi(x) \left[\mu_{\infty}(I) - \left(\mu_{\infty}(I)\right)^2\right]}}\right) = \dfrac{1}{\sqrt{2 \pi}} \int\limits_{-\infty}^{\infty}g(t)e^{-\frac{t^2}{2}}dt.   
\end{equation*}
In other words, for any real numbers $A < B,$ 
$$\lim_{x \to \infty} \text{ Prob }_{1 \leq j \leq r(T)} \left(A < \frac{N_I(j,x) - {\pi(x)}\mu_{\infty}(I)}{\sqrt{\pi(x) \left[\mu_{\infty}(I) - \left(\mu_{\infty}(I)\right)^2\right]}} < B \right) = \frac{1}{\sqrt{2\pi}}\int_A^B e^{-t^2/2}dt.$$
\end{thm}

\subsection{Probabilistic motivation and interpretation}
In order to place Theorem \ref{main-eigenvalues} in the framework of central limit theorems, we may interpret $N_I(f,x)$ as a sum of random variables.  For an even  positive integer $k \geq 2$ and a prime $p,$ we define
$$X_{k,p}(f) := \chi_{I}(a_f(p)),\,f \in \mathcal F_k.$$
Here, $\chi_I$ denotes the characteristic function of the interval $I.$
We now have a double array of random variables $X_{k,p}$ parametrised by the sets $\mathcal F_k$ and primes $p,$ each with expected value, say, $\nu_{k,p}$ and variance $\sigma^2_{k,p}.$ $N_I(f,x)$ can be thought of as the sum of random variables $\sum_{p \leq x} X_{k,p}.$  In the context of central limit-type theorems, it is natural to ask if the random variable
\begin{equation}\label{rv}
\frac{\sum_{p \leq x} \left(X_{k,p} - \nu_{k,p}\right)} {\sqrt{\sum_{p \leq x} \sigma^2_{k,p}}}
\end{equation}
 tends to a normal distribution as $x \to \infty.$   A theorem of Lyapounov \cite[Theorem 27.3]{Billingsley} gives us sufficient conditions for the above to happen.  In our context, we index the rows with weights $k$ and choose $x \leq k$ in each row.  If $X_{k,p}$'s are mutually independent for each $k,$ this theorem of Lyapounov states that if there exists $\delta > 0$ such that
 $$\lim_{ k \to \infty} \frac{\sum_{p \leq x} E\left[\left | X_{k,p} - \nu_{k,p}\right|^{2 + \delta}\right]}{\left(\sum_{p \leq x} \sigma^2_{k,p}\right)^{1+ \frac{\delta}{2}}} = 0,$$
 then the random variable in (\ref{rv}) tends to a normal distribution as $x \to \infty.$  One could show that under appropriate growth conditions on $k$ with respect to $x,$ the above asymptotic holds for $\delta = 2.$  However, we cannot apply Lyapounov's condition since the random variables $X_{k,p}$ are not quite independent. On the other hand, it does give us motivation to investigate whether the sequence (\ref{rv}) tends to a normal distribution under suitable hypothesis. 
  Therefore, Theorem \ref{main-eigenvalues} and its variants can be interpreted as a central limit theorem that holds under additional hypothesis on the growth of $k$ with respect to $x$.

\subsection{Remarks on proofs.} \label{remarks-proof}
Following the spirit of other central limit throems proved in number theory, such as the Erd\"{o}s-Kac theorem, the method of moments proves to be useful.
The main technique used in the proof of Theorem \ref{main-eigenvalues} is the approximation of $N_I(f,x)$ by certain trigonometric polynomials called the Beurling-Selberg polynomials. We then estimate the exponential sums associated to Hecke eigenvalues that arise in these polynomials via the Eichler-Selberg trace formula.  These polynomials were used by M.R. Murty and Sinha \cite{MS} and by Wang \cite{Wang} to obtain error terms in families (B) and (C) respectively.  In this article, we use this technique in a more refined way: we compute moments of functions arising from the Beurling-Selberg polynomials which give approximations to higher moments of $N_I(f,x) - \pi(x)\mu_{\infty}(I).$  The moments of these modified approximating functions are shown to match those of the Gaussian distribution after suitable normalisation.  This refined technique owes its origin to the work of Faifman and Rudnick \cite{FR}, who used it to prove a central limit theorem for the number of zeros of the zeta functions of a family of hyperelliptic curves defined over a fixed finite field as the genus of the curves varies.  The ideas of Faifman and Rudnick have since been fruitfully adapted by various authors (for example, \cite{BDFL},\,\cite{BDFLS},\,\cite{Xiong}) to study similar statistics for different families of smooth projective curves over finite fields. 
  
Nagoshi \cite{Nagoshi} proved another remarkable theorem.  He showed that if $k = k(x)$ satisfes $\frac{\log k}{\log x} \to \infty$ as $x \to \infty,$ then for any bounded continuous real function $h$ on $\R,$ 
$$\lim_{x \to \infty}\frac{1}{|\mathcal F_{k}|}\sum_{f \in \mathcal F_{k}}h\left(\frac{\sum_{p \leq x} a_f(p)}{\sqrt{\pi(x)}} \right) = \frac{1}{\sqrt{2 \pi}} \int\limits_{-\infty}^{\infty}h(t)e^{-\frac{t^2}{2}}dt.$$

In this article, we consider the statistics of $\sum_{p \leq x} \chi_{I} (a_f(p)$ for a fixed interval $I$ as opposed to $\sum_{p\leq x} a_f(p)$ as $f$ is picked up at random from $\mathcal F_k.$ However, we do borrow some combinatorial ideas from Nagoshi's proof in Section \ref{odd-m} of this paper.

The proofs of Theorems \ref{main-eigenvalues-harmonic} and \ref{main-eigenvalues-Maass} are very similar to that of Theorem \ref{main-eigenvalues}.  The key difference is in the trace formulas used to estimate the exponential sums arising from the Beurling-Selberg approximation for these families.  Hence, we shall omit the proofs.  For Theorem \ref{main-eigenvalues-harmonic}, one uses a trace formula of Petersson (see \cite[Section 2]{ILS}).  For the case of Maass forms, one uses an unweighted version of the Kuznetsov trace formula. This has been derived by Lau and Wang (\cite[Lemma 3.3]{LW}).   We also make a note that we do not assume the Ramanujan-Petersson conjecture in Theorem \ref{main-eigenvalues-Maass}.  Therefore, in treating the case of Maass forms, we have to take adequate care of the contribution of the ``exceptional" eigenvalues $a_j(p),$ that is, those eigenvalues which could possibly lie outside the interval $[-2,2].$  This is done with the help of a result of Sarnak (\cite[Theorem 1]{Sarnak}) which estimates the density of such exceptional eigenvalues.

\subsection{Outline} In Section \ref{preliminaries}, we set up some notation and review some important properties of Hecke eigenvalues that will be needed in the proof of Theorem \ref{main-eigenvalues}.  In Section \ref{S-B-technique}, we describe the Beurling-Selberg polynomials and prove some results about the asymptotics of their Fourier coefficients.  In Section \ref{1-m}, we use the Beurling-Selberg polynomials to derive the expected value of $N_I(f,x)$ for $f \in \mathcal F_k$ and obtain error terms in the theorem of Nagoshi.  In Section \ref{2-m}, we derive the second central moment of $N_I(f,x)$. In Section \ref{convergence-distribution}, we describe the strategy for the proof of Theorem \ref{main-eigenvalues}.  We show that in order to prove Theorem \ref{main-eigenvalues}, it is sufficient to derive the higher odd and even moments of our modified approximating functions for $N_I(f,x) - \pi(x)\mu_{\infty}(I).$  In Section \ref{odd-m}, we derive these higher moments and deduce Theorem \ref{main-eigenvalues}.

\section{Preliminaries} \label{preliminaries}
In this section, we state fundamental results about modular forms and eigenvalues of Hecke operators that will be needed in the proof of Theorem \ref{main-eigenvalues}.  We start by recalling the following classical lemma which describes multiplicative relations between $a_f(p)$'s.

\begin{lem}\label{Hecke-multiplicative}
Let $f \in \mathcal F_k.$ For primes $p_1,\,p_2 $ and non-negative integers $i,\,j,$ 
$$ a_f(p_1^i)a_f(p_2^j) = 
\begin{cases}
a_f(p_1^ip_2^j) &\text{ if }p_1 \neq p_2\\
\sum_{l = 0}^{\min{(i,j)}}a_f(p_1^{i + j - 2l}) &\text{ if }p_1 = p_2.
\end{cases}$$

\end{lem}
The recursive relations between $a_f(p^m)$'s for $m \geq 0$ can be elegantly encoded by the following lemma \cite[Lemma 1]{Serre}.
\begin{lem}\label{Chebyshev}
For a prime $p$ and $f \in \mathcal F_k,$ let $\thet$ be the unique angle in $[0,\pi]$ such that $a_f(p) = 2\cos \thet.$ For $m\geq 0,$
$$a_f(p^m) = X_m(a_f(p)),$$
where the $m$-th Chebyshev polynomial is defined as follows:
$$X_m(x) = \frac{\sin (m+1) \theta}{\sin \theta},\,x = 2\cos\theta.$$
\end{lem}
We observe that for $m \geq 2,$
$$2\cos m\theta = X_m(2\cos \theta) - X_{m-2}(2\cos \theta).$$
Thus, 
we have the following corollary to the above lemma.
\begin{cor}\label{cos-m}
With the same notation as in Lemma \ref{Chebyshev}, for $m \in \Z,\,m\neq 0,$
$$2 \cos (m \thet) =
\begin{cases} a_f(p)&\text{ if } |m| = 1\\
a_f(p^{|m|}) - a_f(p^{|m|-2})&\text{ if }|m| \geq 2.
\end{cases}
$$
\end{cor}

\begin{prop}\label{trace}
Let $k$ be a positive even integer and $n$ be a positive integer.  We have
$$\sum_{f \in \mathcal F_k} a_f(n) = 
\begin{cases}
\frac{k-1}{12}\left(\frac{1}{\sqrt{n}}\right) + \O\left(n^c\right)&\text{ if }n\text{ is a square }\\
\O\left(n^c\right)&\text{ otherwise.}
\end{cases}
$$
Here, $c=\frac{1}{2} + \varepsilon$ and the implied constant in the error term is absolute.  
\end{prop}
\begin{proof}
This proposition follows from the Eichler-Selberg trace formula for Hecke operators $T_n,\,n\geq 1$ acting on $S_k.$  The Eichler-Selberg trace formula (see \cite[Sections 7, 8]{MS} and \cite[Section 4]{Serre}) states that for every integer $n \geq 1,$
$$\sum_{f \in \mathcal F_k} a_f(n) = \sum_{i=1}^4 B_i(n),$$
where $B_i(n)$'s are as follows:
$$
B_1(n) = \begin{cases}
\frac{k-1}{12}\left(\frac{1}{\sqrt{n}}\right) &\text{ if }n\text{ is a square }\\
0 &\text{ otherwise.}
\end{cases}
$$
$$B_2(n) = -\frac{1}{2}\frac{1}{n^{(k-1)/2}}\sum_{t \in \Z,\,t^2 < 4n}\frac{\varrho^{k-1}- \overline{\varrho}^{k-1}}{\varrho - \overline{\varrho}}H(4n-t^2).$$
Here, $\varrho$ and $\overline{\varrho}$ denote the zeroes of the polynomial $x^2 - tx+n$ and for a positive integer $l,$ $H(l)$ denotes the Hurwitz class number.  
$$B_3(n) = - \frac{1}{n^{(k-1)/2}}\sum_{d|n \atop {0 \leq d \leq \sqrt n}}^{(b)} d^{k-1}.$$
The notation $(b)$ on top of the summation denotes that if there is a contribution from $d = \sqrt n,$ it should be multiplied with $1/2.$  Finally,
$$B_4(n) = \begin{cases}
\frac{1}{n^{(k-1)/2}}\sum_{d|n} d &\text{ if }k = 2,\\
0 &\text{ otherwise.}
\end{cases}
$$
To estimate $B_2(n),$ we observe that $|\varrho| = \sqrt n.$  Thus,
$$\left|\frac{\varrho^{k-1} - \overline{\varrho}^{k-1}}{\varrho - \overline{\varrho}}\right| \leq \frac{2n^{(k-1)/2}}{\sqrt{4n-t^2}}.$$
Following a classical estimate of Hurwitz, we have
$$H(4n - t^2) \ll \sqrt{4n-t^2}\log ^2(n),$$
the implied constant being absolute.
Thus,
$$|B_2(n)| \ll \sqrt n \log^2 n.$$
One can immediately observe that
$$|B_3(n)| \ll \sum_{d|n  \atop {d \leq \sqrt{n}}} 1$$
and
$$|B_4(n)| \ll \sqrt n \sum_{d|n}1.$$
Combining the above estimates, we prove Proposition \ref{trace}.
\end{proof}
In particular, $n=1$ in the above trace formula gives us
\begin{equation}\label{dimension-asymptotics}
s_k = \frac{k-1}{12} + \O(1).
\end{equation}

We also record the following important estimate:
\begin{equation}\label{primes-reciprocals}
\sum_{p \leq x}\frac{1}{p} = \O(\log \log x).
\end{equation}

In particular, using Proposition \ref{trace}, we have the following lemma:
\begin{lem}\label{trace-asymp}
Suppose $k = k(x)$ runs over positive even integers such that $\frac{\log k}{\log x} \to \infty$ as $x \to \infty.$  Then, for any positive integer $m$ and and positive real number $a,$ we have
\begin{equation}\label{name1} 
\lim_{x \to \infty}  \frac{1}{(\pi(x))^a s_k}\sum_{p\leq x} \sum_{f \in \mathcal F_k} a_f(p^m) = 0.
\end{equation}
Furthermore, for non-negative integers $m_1,\,m_2,\dots m_r$ not all zero, 
\begin{equation}\label{name2} 
\lim_{x \to \infty}  \frac{1}{(\pi(x))^a s_k}\sum_{p_1,\,p_2,\dots p_r \leq x} \sum_{f \in \mathcal F_k} a_f(p_1^{m_1}p_2^{m_2}\dots p_r^{m_r}) = 0,
\end{equation}
where $p_1,\,p_2,\dots p_r$ are distinct primes not exceeding $x.$
\end{lem}

\begin{proof}
From Proposition \ref{trace}, equations (\ref{dimension-asymptotics}) and (\ref{primes-reciprocals}), one deduces, for $m \geq 1,$ the following:
\begin{equation}\label{sum-trace}
\begin{split}
&\frac{1}{s_k}\sum_{p\leq x} \sum_{f \in \mathcal F_k} a_f(p^m) 
=\begin{cases}
 \sum_{p\leq x}\left(\frac{1}{p^{\frac{m}{2}}} + \O\left(\frac{p^{mc}}{k}\right)\right)&\text{ if }m \text{ is even}\\
\O\left(\sum_{p\leq x}\frac{p^{cm}}{k}\right)&\text{ if }m\geq 1,\,m \text{ is odd}.
\end{cases}\\
&=\begin{cases}
\O(\log \log x) + \O\left(\frac{x^{cm}\pi(x)}{s_k}\right)&\text{ if }m = 2\\
\O(1) + \O\left(\frac{x^{cm}\pi(x)}{s_k}\right)&\text{ if }m > 2,\,m \text{ is even}\\
\O\left(\frac{x^{cm}\pi(x)}{s_k}\right)&\text{ if }m\geq 1,\,m \text{ is odd}.
\end{cases}
\end{split}
\end{equation}
Since $\frac{\log k}{\log x} \to \infty$ as $x \to \infty,$ 
$$\lim_{x \to \infty} \frac{x^r}{k} = 0$$
for any real power $r > 0.$   Moreover,
$$\log \log x = \o(\pi(x))^a,$$
for any $a > 0.$   This proves equation (\ref{name1}).  Equation (\ref{name2}) follows by a similar argument.
\end{proof}

\begin{remark}
We note that the proof outlined above gives us a stronger statement, which is of independent interest.  Let us assume the same growth conditions on $k$ as stated above.  Equation (\ref{sum-trace}) tells us that for any $a>1,$ with , 
\begin{equation*} 
\lim_{x \to \infty}  \frac{1}{(\log \log x)^a s_k}\sum_{p\leq x} \sum_{f \in \mathcal F_k} a_f(p^m) = 0.
\end{equation*}
Furthermore, for non-negative integers $m_1,\,m_2,\dots m_r$ not all zero, 
\begin{equation*} 
\lim_{x \to \infty}  \frac{1}{(\log \log x)^a s_k}\sum_{p_1,\,p_2 \dots p_r \leq x} \sum_{f \in \mathcal F_k} a_f(p_1^{m_1}p_2^{m_2}\dots p_r^{m_r}) = 0,
\end{equation*}
where $p_1,\,p_2,\dots p_r$ are distinct primes not exceeding $x.$
\end{remark}
\section{Beurling-Selberg polynomials} \label{S-B-technique}

The Beurling-Selberg polynomials are trigonometric polynomials which provide a good approximation to the characteristic functions of intervals in $\R.$  The strength of these polynomials is that they reduce the estimation of counting functions to evaluating finite exponential sums.  We briefly review important properties of these polynomials in this section and refer the reader to a detailed exposition by Montgomery (see \cite[Chapter 1]{Mont}).

Let $I = [\alpha,\beta] \subseteq [-\frac{1}{2},\frac{1}{2}]$ and $M \geq 1$ be an integer.  One can construct trigonometric polynomials $S_M^-(x)$ and $S_M^+(x)$ of degree less than or equal to $M,$ respectively called the minorant and majorant Beurling-Selberg polynomials for the interval $I$ such that
\begin{itemize}
\item[(a)] For all $x \in \R,\,S_M^-(x)\leq \chi_I(x) \leq S_M^+(x)$
\item[(b)] $$\int_{-1/2}^{1/2} S_M^{\pm}(x) dx = \beta - \alpha \pm \frac{1}{M+1},$$
\item[(c)] For $ 0 < |m| \leq M,$
\begin{equation}\label{FC}
\left|\st(m) - \widehat{\chi}_I(m)\right| \leq \frac{1}{M+1}.
\end{equation}
\end{itemize}
Henceforth, we will use the following notation: 
for an interval $I=[a,b] \subseteq [-2,2],$ we choose a subinterval $$I_1 = [\alpha,\beta] \subseteq \left[0,\frac{1}{2}\right ]$$
such that
$$\theta \in I_1 \iff 2\cos (2\pi \theta) \in I.$$
For $M \geq 1,$ let $S^{\pm}_{M,1}(x)$ denote the majorant and minorant Beurling-Selberg polynomials for the interval $I_1.$
We denote, for $0 \leq |m| \leq M,$
$$\A(m) = \hat{S}^{\pm}_{M,1}(m) + \hat{S}^{\pm}_{M,1}(-m).$$
By equation (\ref{FC}), we have, for $1 \leq |m| \leq M,$
$$ \hat{S}^{\pm}_{M,1}(m) = \widehat{\chi}_I(m) + \O\left(\frac{1}{M+1}\right) = \frac{e(-m\alpha) - e(-m\beta)}{2 \pi i m} + \O\left(\frac{1}{M+1}\right)$$
Thus,
\begin{equation}\label{AFC}
\A(m) = \widehat{\chi}_I(m) + \widehat{\chi}_I(-m) + \O\left(\dfrac{1}{M+1}\right) = \frac{\sin(2\pi m\beta)-\sin(2\pi m\alpha)}{m\pi} + \O\left(\dfrac{1}{M+1}\right).
\end{equation}

For $M\geq 3$ and $1 \leq m \leq M,$ let
\begin{equation}\label{U-defn}
\U(m) :=
\begin{cases}
\A(m)-\A(m+2),&\text{ if }1 \leq m \leq M-2\\
\A(m),&\text{ if }m = M-1,\,M.
\end{cases}
\end{equation}

We record the following bound, which is not optimal, but good enough for our purposes.
\begin{lem}\label{Um-bounds}
Let $I = [\alpha,\beta]$ be a fixed interval and $\underline{m}_r=(m_1,\ldots,m_r)$ be an $r$-tuple of positive integers such that $1\leq m_i\leq M.$ Let $\U(\underline{m}_r)=\U(m_1)\cdots\U(m_r)$. 
$$\sum_{\underline{m}_r}^{(3)} |\U(\underline{m}_r)| = \O_r(\log M)^{r}.$$
Here, $\sum\limits_{\underline{m}_r}^{(3)}$ denotes that the sum is taken over $r$-tuples of positive integers lying between 1 and $M.$
\end{lem}

\begin{proof}
From equation (\ref{FC}), we observe that for any $1 \leq m \leq M,$
$$|\U(m)| \leq \frac{2}{\pi|m|} + \frac{2}{M+1}.$$
Thus, for a fixed $r$-tuple $(\underline{m}_r),$
\begin{equation*}
\begin{split}
&|\U(\underline{m}_r)| \ll \prod_{j=1}^{r}\left(\frac{1}{m_{j}} + \frac{1}{M+1}\right)\\
&\ll \frac{1}{(M+1)^k} +\sum_{k = 1}^r \frac{1}{(M+1)^{r-k}}\sum_{j_1,j_2,\dots, j_k = 1}^r\frac{1}{m_{j_1}m_{j_2}\cdots m_{j_k}}.
\end{split}
\end{equation*}
Hence,
\begin{equation*}
\begin{split}
&\sum_{\underline{m}_r}^{(3)}|\U(\underline{m}_r)|\ll \sum_{\underline{m}_r}^{(3)}\left(\frac{1}{(M+1)^k} +\sum_{k = 1}^r \frac{1}{(M+1)^{r-k}}\sum_{j_1,j_2,\dots, j_k = 1}^r\frac{1}{m_{j_1}m_{j_2}\cdots m_{j_k}}\right)\\
&\ll \sum_{\underline{m}_r}^{(3)}\frac{1}{(M+1)^k} + \sum_{k = 1}^{r}\frac{1}{(M+1)^{r-k}}\sum_{\underline{m}_r}^{(3)}\sum_{j_1,j_2 \dots, j_k = 1}^r \frac{1}{m_{j_1}m_{j_2}\cdots m_{j_k}}\\
& \ll\sum_{k=0}^{r}\binom{r}{k}\frac{1}{(M+1)^{r-k}}M^{r-k}(\log M)^k \ll_r(\log M)^{r}.
\end{split}
\end{equation*}

\end{proof}

\section{First moment}\label{1-m}
For an interval $I=[a,b] \subseteq [-2,2],$ we define
$$N_I(f,x) := \#\left\{p \leq x:\,a_f(p) \in I\right\}.$$
Denoting $a_f(p) = 2\cos \thet,\,$ with $ \thet \in [0,\pi],$ we consider the families
$$\left\{\frac{\thet}{2\pi},\,-\frac{\thet}{2\pi},\,f \in \mathcal F_k\right\}.$$
As before, we choose a subinterval $$I_1 = [\alpha,\beta] \subseteq \left[0,\frac{1}{2}\right ]$$
so that 
$$\frac{\thet}{2\pi} \in I_1 \iff 2\cos \thet \in I.$$
We denote $I_2 = (\alpha,\beta].$
Thus,
\begin{equation*} 
N_I(f,x) = \sum_{p\leq x}\left[\chi_{I_1}\left(\frac{\thet}{2\pi}\right) + \chi_{I_2}\left(-\frac{\thet}{2\pi}\right)\right],
\end{equation*}
since 
$$\chi_{I_2}\left(-\frac{\thet}{2\pi}\right) = 0.$$
Following the notation and properties of the Beurling-Selberg polynomials from the previous section, we have
\begin{equation}\label{I1}
\sum_{p \leq x}\left[S^-_{M,1}\left(\frac{\thet}{2\pi}\right) + S^-_{M,1}\left(-\frac{\thet}{2\pi}\right)\right] \leq N_I(f,x) \leq \sum_{p\leq x}\left[S^+_{M,1}\left(\frac{\thet}{2\pi}\right) + S^+_{M,1}\left(-\frac{\thet}{2\pi}\right)\right].
\end{equation}
 Our aim is to compute, for every positive integer $r,$
$$\lim_{x \to \infty}\frac{1}{s_k}\sum_{f \in \mathcal F_k} \left(N_I(f,x) - \pi(x)\mu_{\infty}(I)\right)^r.$$ 
Our strategy is to use equation (\ref{I1}) to approximate $N_I(f,x) - \pi(x)\mu_{\infty}(I)$ by certain trigonometric polynomials and evaluate the moments of these polynomials.

We observe
\begin{equation} \label{N_I(f,x)}
\begin{split}
N_I(f,x) & \leq \sum_{p\leq x}\left[S_{M,1}^+\left(\tht\right) + S_{M,1}^+\left(-\tht\right) \right]\\
&=\sum_{p\leq x}\sum\limits_{|m|\leq M}\left[\hat{S}^{+}_{M,1}(m)\left\{ e\left(\frac{m\thet}{2\pi}\right) + e\left(-\frac{m\thet}{2\pi}\right) \right\} \right]\\
& = \sum_{p\leq x} \sum\limits_{|m|\leq M} \left[\hat{S}^{+}_{M,1}(m)\left(2\cos (m\thet)\right)\right]\\
&= 2\sum_{p\leq x} \hat{S}^{+}_{M,1}(0) + \sum\limits_{m=1}^{M}(\hat{S}^{+}_{M,1}(m)+ \hat{S}^{+}_{M,1}(-m))\sum_{p \leq x} 2\cos( m\thet)\\
&= \pi(x) \Ap(0) +  \sum\limits_{m=1}^{M}\Ap(m)\sum_{p \leq x} 2\cos( m\thet).
\end{split}
\end{equation}
By a similar argument, we derive
\begin{equation}\label{N_I(f,x)lower}
N_I(f,x) \geq  \pi(x) \hat{\mathcal{S}}^{-}_{\text{M}}(0) +  \sum\limits_{m=1}^{M}\hat{\mathcal{S}}^{-}_{\text{M}}(m)\sum_{p \leq x} 2\cos(m\thet).
\end{equation}

Let us denote
\begin{equation*}
S^{\pm}(M,f)(x) := \sum_{m=1}^2\A(m)\sum_{p \leq x} a_f(p^m) +\sum\limits_{m=3}^{M}\A(m)\sum_{p \leq x} \left(a_f(p^{m}) - a_f(p^{m-2})\right).
\end{equation*}
By combining equations (\ref{N_I(f,x)}), (\ref{N_I(f,x)lower}) and Corollary  \ref{cos-m}, we get
\begin{equation}\label{Upper}
N_I(f,x) - \pi(x)\left[\Ap(0) - \Ap(2)\right] \leq S^{+}(M,f)(x) 
\end{equation}
and 
\begin{equation}\label{Lower}
 S^{-}(M,f)(x) \leq N_I(f,x) - \pi(x)\left[\hat{\mathcal{S}}^{-}_{\text{M}}(0) - \hat{\mathcal{S}}^{-}_{\text{M}}(2)\right].
\end{equation}

We are now ready to calculate the first moment of $N_I(f,x).$  Henceforth, for any function $\phi:\,S_k \to \C,$ we denote the average
$$\langle \phi(f)\rangle := \frac{1}{s_k}\sum_{f \in \mathcal F_k} \phi(f).$$
In order to derive the moments $\langle (X_f(x))^r \rangle,$ we explore the moments of $S^{\pm}(M,f)(x).$  In this direction, we state the following proposition, which shows that the Sato-Tate conjecture is true on average as $x \to \infty.$
 
\begin{prop}\label{SatoTateAverage}
Let $k = k(x)$ be a positive even integer. Then, for any interval $I = [a,b] \subseteq [-2,2],$
$$\left\langle{N_I(f,x)}\right \rangle = \pi(x)\int_{a}^{b} \mu_{\infty}(t)dt + \O\left(\frac{\pi(x)\log x}{\log k} + \log\log x\right). $$
Thus, if $k = k(x)$ runs over positive even integers such that $\frac{\log k}{\log x} \to \infty$ as $x \to \infty,$ then
$$\lim_{x \to \infty} \dfrac{1}{\pi(x)}\left\langle{N_I(f,x)}\right \rangle = \int_{a}^{b} \mu_{\infty}(t)dt.$$

\end{prop}
\begin{remark}
Proposition \ref{SatoTateAverage} is essentially due to Y. Wang \cite[Theorem 1.1]{Wang}.  He proves an analogous result for primitive Maass forms and indicates that a similar technique works for the average Sato-Tate family.  We provide a brief proof of this proposition as a first step in evaluating moments of the polynomials $S^{\pm}(M,f)(x).$
\end{remark}
\begin{proof}
We have, by equation (\ref{AFC}),
\begin{equation*} 
\A(2) = \frac{\sin 4\pi\beta - \sin 4\pi\alpha}{2\pi} + \O\left(\frac{1}{M+1}\right)
\end{equation*}
and $$\A(0) = 2(\beta -\alpha) \pm \dfrac{2}{M+1}.$$
Combining the above with equations (\ref{Upper}) and (\ref{Lower}), we can find constants $C$ and $D$ such that
\begin{equation}\label{SandNT}
\begin{split}
S^{-}{(M,f)}(x) + C\left(\frac{\pi(x)}{M+1}\right)\
&\leq  N_I(f,x) - \pi(x)\left((2\beta-2\alpha) - \frac{\sin 4\pi\beta - \sin 4\pi\alpha}{2\pi}\right)\\
& \leq S^{+}{(M,f)}(x) + D\left(\frac{\pi(x)}{M+1}\right).
\end{split}
\end{equation}
We observe, for $[\alpha,\beta] \in [0,1/2]$ as chosen before,
\begin{equation*} 
\begin{split}
(2\beta-2\alpha) - \frac{\sin 4\pi\beta - \sin 4\pi\alpha}{2\pi} &= 2\int_{\alpha}^{\beta} (1 - \cos 4\pi \theta) d\theta\\
&= 4\int_{\alpha}^{\beta} \sin^2 2\pi\theta d\theta\\
&=\int_a^b \mu_{\infty}(t) dt.
\end{split}
\end{equation*}
Thus, for every positive integer $M,$
\begin{equation}\label{STeffective}
 S^{-}(M,f)(x) + C\left(\frac{\pi(x)}{M+1}\right) \leq N_I(f,x) - \pi(x)\int_a^b \mu_{\infty}(t) dt \leq S^{+}(M,f)(x) + D\left(\frac{\pi(x)}{M+1}\right).
 \end{equation}
 By equation (\ref{sum-trace}),
 $$\left\langle S^{\pm}(M,f)(x)\right\rangle = \O\left(\sum_{m=1\atop{m \text{ even }}}^M|\A(m)|\sum_{p\leq x}\begin{cases} \frac{1}{p}, &\text{ if }m = 2\\
 \frac{1}{p^{\frac{m}{2}-1}}, &\text{ if }m\geq 4
 \end{cases}\right) + \O\left(\sum_{m=1}^M|\A(m)|\frac{x^{cm}\pi(x)}{k}\right).$$
 Since $ |\A(m)| \ll \frac{1}{m},$ we get, for every positive integer $M,$
 $$\sum_{m=1}^M|\A(m)|\frac{x^{cm}\pi(x)}{k} =\O\left(\frac{\pi(x)}{k}\sum_{m=1}^M\frac{x^{cm}}{m}\right) = \O\left(\frac{\pi(x)x^{cM}}{k}\right).$$
 Thus,
 \begin{equation*}
 \begin{split}
 \left\langle S^{\pm}(M,f)(x)\right\rangle \
 &=\O\left(\sum_{p \leq x}\left(\frac{1}{p} + \sum_{m=2}^{\infty}\frac{1}{p^m}\right) + \sum_{m=1}^M|\A(m)|\frac{x^{cm}\pi(x)}{k}\right)\\
 &= \O\left(\log \log x + \frac{x^{cM}\pi(x)}{k}  \right).
 \end{split}
 \end{equation*}
 That is, for every positive integer $M,$ by equation (\ref{STeffective}), we have
 $$\left\langle N_I(f,x) - \pi(x)\int_a^b \mu_{\infty}(t) dt \right\rangle= \O\left(\log \log x + \frac{x^{cM}\pi(x)}{k} + \frac{\pi(x)}{M+1}\right).$$
 
 We now choose
 $$M = \left\lfloor\frac{d\log k}{c\log x}\right\rfloor$$
  for some $0 < d <1.$
 This proves the proposition.
 \end{proof}
 
 \section{Second moment}\label{2-m}
 In this section, we will compute 
 $$\lim_{x \to \infty}\frac{1}{\pi(x)}\left\langle (S^{\pm}(M,f)(x))^2)\right\rangle.$$ 
 \begin{prop}\label{SecondMomentApprox}
 	Let $[\alpha,\beta]$ be a fixed interval in $[0,1/2].$  Then, for every $M \geq 3,$
 	\begin{equation}\label{SquaresVM}
 	\left\langle (S^{\pm}(M,f)(x))^2)\right\rangle
 	=\pi(x)\left(\mu_{\infty}(I)-\mu_{\infty}(I)^2\right) +\O\left((\log \log x)^2 + \frac{\pi(x)^2 x^{2Mc}}{k} + \dfrac{\pi(x)}{M+1}\right).
 	\end{equation} 
 \end{prop}
 \begin{proof}
 	We denote $S^{\pm}_M(x) := S^{\pm}_{M,1}(x) + S^{\pm}_{M,1}(-x).$  Recall that $S^{\pm}(M,f)(x)$ was obtained after removing $\pi(x)(\A(0)-\A(2))$ from the Fourier expansion of $\sum_{p \leq x}S^{\pm}_M\left(\dfrac{\theta_f(p)}{2\pi}\right)$. Therefore we may write $$ S^{\pm}(M,f)(x)= \sum_{p \leq x}S^{\pm}_M\left(\dfrac{\theta_f(p)}{2\pi}\right) - \pi(x)(\A(0)-\A(2)).$$
 	Squaring both sides, the following expansion is obtained.
 	\begin{equation}\label{2ndMomentExpansion}
 	\begin{split}
 	S^{\pm}(M,f)(x)^2 &= \left(\sum_{p \leq x}S^{\pm}_M\left(\dfrac{\theta_f(p)}{2\pi}\right) \right)^2 - 2\pi(x)(\A(0)-\A(2))\sum_{p \leq x}S^{\pm}_M\left(\dfrac{\theta_f(p)}{2\pi}\right) + \pi(x)^2(\A(0)-\A(2))^2\\
 	&= \left(\sum_{p \leq x}S^{\pm}_M\left(\dfrac{\theta_f(p)}{2\pi}\right) \right)^2 -2\pi(x)(\A(0)-\A(2))\left(S^{\pm}(M,f)(x) + \pi(x)(\A(0)-\A(2))\right) \\
 	& \quad + \pi(x)^2(\A(0)-\A(2))^2\\
 	&= \left(\sum_{p \leq x}S^{\pm}_M\left(\dfrac{\theta_f(p)}{2\pi}\right) \right)^2 -2\pi(x)(\A(0)-\A(2))S^{\pm}(M,f)(x) - \pi(x)^2(\A(0)-\A(2))^2.
 	\end{split}
 	\end{equation}
 	We have:
 	\begin{equation}\label{total-square}
 	\left(\sum_{p \leq x}S^{\pm}_M\left(\dfrac{\theta_f(p)}{2\pi}\right) \right)^2  = \sum_{p \leq x} {S^{\pm}_M}^2\left(\dfrac{\theta_f(p)}{2\pi}\right) + \sum_{{p,q\leq x}\atop {p\neq q}}S^{\pm}_M\left(\dfrac{\theta_f(p)}{2\pi}\right)S^{\pm}_M\left(\dfrac{\theta_f(q)}{2\pi}\right).
 	\end{equation}
 	First we consider the sum $$\sum_{p \leq x} {S^{\pm}_M}^2\left(\dfrac{\theta_f(p)}{2\pi}\right).$$ Writing out the Fourier expansion
	$$ {S^{\pm}_M}^2\left(\frac{\theta_f(p)}{2\pi}\right) = \sum_{m = 0}^{2M} \T(m)(2 \cos (m\theta_f(p))),$$
	 we obtain the following:
 	\begin{equation}\label{S_M^2}
 	\begin{split}
 	\sum_{p \leq x} {S^{\pm}_M}^2\left(\dfrac{\theta_f(p)}{2\pi}\right) &= \sum_{p\leq x}\sum_{m=0}^{2M}\T(m)(2\cos(m\theta_f(p)))\\
 	&=\pi(x)\left(\T(0) - \T(2)\right) + \sum_{p \leq x}\left( \sum_{m = 1}^{2}\T(m)a_f(p^m) + \sum_{m=3}^{2M}\T(m)(a_f(p^m)- a_f(p^{m-2})\right).
 	\end{split}
 	\end{equation}
 	Observe that for $0\leq m\leq 2M$,
 	\begin{equation*}
 	\begin{split}
 	|\T(m) - \widehat{\chi_I^2}(m)| &\leq \int_{0}^{1} |{S^{\pm}_M}^2(x) - {\chi_I}^2(x)|dx\\
 	&=\int_{0}^{1} |S^{\pm}_M(x) -\chi_I(x)| |S^{\pm}_M(x) + \chi_I(x)| dx\\
 	&= \O\left(\dfrac{1}{M+1}\right)
 	\end{split}
 	\end{equation*}
 	since $S^{\pm}_M(x)$ and $\chi_I(x)$ are bounded and $|| S^{\pm}_M(x) -\chi_I(x)||_1 \leq \dfrac{1}{M+1}.$ Therefore 
 	$$\T(m) = \widehat{\chi_I^2}(m) + \O\left(\dfrac{1}{M+1}\right)=\widehat{\chi}_I(m) + \O\left(\dfrac{1}{M+1}\right), $$ since $\chi_I^2(x) = \chi_I(x)$ for all $x \in \R$. 
  Applying this to equation (\ref{S_M^2}), we see that 
 	\begin{equation*}
 	\begin{split}
 	\left \langle\sum_{p \leq x} {S^{\pm}_M}^2\left(\dfrac{\theta_f(p)}{2\pi}\right)\right\rangle &= \pi(x)\mu_{\infty}(I) + \frac{1}{s_k}\sum_{p \leq x}\left( \sum_{m = 1}^{2}\T(m)\sum_{f \in \mathcal F_k}a_f(p^m) + \sum_{m=3}^{2M}\T(m) \sum_{f \in \mathcal F_k}(a_f(p^m)- a_f(p^{m-2})\right) \\
 	&\quad +  \O\left(\dfrac{\pi(x)}{M+1}\right)
 	\end{split} 	
 	\end{equation*}
 	Moreover, since $\T(m) \ll \frac{1}{m},$ using the trace formula (as in equation (\ref{name2})), the following holds:
 	\begin{equation}\label{S_M^2final}
 	\dfrac{1}{s_k}\sum_{f \in \mathcal F_k}\sum_{p \leq x} {S^{\pm}_M}^2\left(\dfrac{\theta_f(p)}{2\pi}\right) = \pi(x)\mu_{\infty}(I) + \O\left(\log\log x + \dfrac{x^{2cM}\pi(x)}{k} + \dfrac{\pi(x)}{M+1}\right).
 	\end{equation}
 	$$$$
 	
 	Now we analyze the term $$\sum_{{p,q\leq x}\atop {p\neq q}}S^{\pm}_M\left(\dfrac{\theta_f(p)}{2\pi}\right)S^{\pm}_M\left(\dfrac{\theta_f(q)}{2\pi}\right).$$
 	It is easy to see that 
 	\begin{equation*}
 	\begin{split}
 	\sum_{{p,q\leq x}\atop {p\neq q}}S^{\pm}_M\left(\dfrac{\theta_f(p)}{2\pi}\right)S^{\pm}_M\left(\dfrac{\theta_f(q)}{2\pi}\right) &= (\pi(x)^2 -\pi(x))(\A(0)-\A(2))^2 + 2\pi(x)(\A(0)-\A(2))S^{\pm}(M,f)(x)\\
 	&\quad + \sum_{m_1,\,m_2 = 1}^{M}\U(m_1)\U(m_2)\sum_{{p,q\leq x}\atop {p\neq q}}a_f(p^{m_1})a_f(q^{m_2}),
 	\end{split}
 	\end{equation*}
 	where  $\sum\limits_{{p,q\leq x}\atop {p\neq q}} 1= \pi(x)^2 -\pi(x)$ and $\U(m)$ is as defined in equation (\ref{U-defn}).
 	Again, using the trace formula and a calculation similar to equation (\ref{sum-trace}) it is not hard to show that 
 	\begin{equation}\label{2ndMoment-error}
 \dfrac{1}{s_k}\sum_{f\in\F_k} \sum_{m_1,\,m_2 = 1}^{M}\U(m_1)\U(m_2)\sum_{{p,q\leq x}\atop {p\neq q}}a_f(p^{m_1})a_f(q^{m_2}) =\O\left((\log\log x)^2 + \dfrac{\pi(x)^2x^{2Mc}}{k}\right).
 	\end{equation}
Therefore,
\begin{equation}\label{S_M(p)S_M(q)}
\begin{split}
\dfrac{1}{s_k}\sum_{f \in \mathcal \F_{k}}\sum_{{p,q\leq x}\atop {p\neq q}}S^{\pm}_M\left(\dfrac{\theta_f(p)}{2\pi}\right)S^{\pm}_M\left(\dfrac{\theta_f(q)}{2\pi}\right) &= (\pi(x)^2 -\pi(x))(\A(0)-\A(2))^2 + 2\pi(x)(\A(0)-\A(2))\langle S^{\pm}(M,f)(x)\rangle \\
&\quad+ \O\left((\log\log x)^2 + \dfrac{\pi(x)^2x^{2Mc}}{k}\right).
\end{split}
\end{equation}
 	We now write $(\A(0)-\A(2)) = \mu_{\infty}(I)+ \O\left(\dfrac{1}{M+1}\right)$ and use equations (\ref{S_M^2final}) and (\ref{S_M(p)S_M(q)}) in (\ref{total-square}) and (\ref{2ndMomentExpansion}) to get the following:
 	
 	$$\dfrac{1}{s_k}\sum_{f\in\F_k} {S^{\pm}(M,f)(x)^2} = \pi(x)(\mu_{\infty}(I) - \mu_{\infty}(I)^2) + \O\left( {(\log\log x)^2} + \dfrac{\pi(x)^2x^{2Mc}}{k} + \dfrac{\pi(x)}{M+1}\right).$$
 \end{proof} 	
 	In conclusion, we have
 	$$\lim_{x \to \infty}\frac{1}{\pi(x)}\left\langle (S^{\pm}(M,f)(x))^2)\right\rangle =  \mu_{\infty}(I) - \mu_{\infty}(I)^2,$$ if we let $M= M(x)$ and $k= k(x)$ to grow appropriately with respect to $x$ so that the error term is negligible.

\begin{remark}
By almost exactly the same process, one can show that
\begin{equation}\label{plusminusVM}
\begin{split}
&\left\langle S^+(M,f)(x) S^-(M,f)(x)\right\rangle = \pi(x)\left(\mu_\infty(I) -\mu_\infty(I)^2\right)\\
&+ \O\left((\log \log x)^2 + \frac{\pi(x)^2 x^{2Mc}}{k} + \dfrac{\pi(x)}{M+1}\right).
\end{split}
\end{equation}
\end{remark}
 
 \bigskip
 
 \section{Strategy for proof of main theorem}\label{convergence-distribution}
The proof of Theorem \ref{main-eigenvalues} depends on the following fundamental steps.

\bigskip

\begin{itemize}

\item[{\bf(1)}] We first show that for a suitable choice of $M = M(x),$
$$\frac{S^{\pm}(M,f)(x)}{\sqrt{\pi(x)(\mu_{\infty}(I) - \mu_{\infty}(I)^2))}}$$
converges in mean square to
$$\frac{N_I(f,x) - \pi(x)\mu_{\infty}(I)}{\sqrt{\pi(x) (\mu_{\infty}(I) - \mu_{\infty}(I)^2))}}$$
as $x \to \infty.$  This forms the content of Proposition \ref{mean-squares}.
\begin{remark}
	This convergence holds as we vary the families $\mathcal F_k$ under certain growth conditions on $k.$ As will be seen in equation (\ref{Msquare}), this convergence holds if $M$ grows faster than $\sqrt{\pi(x)}$ and we impose appropriate growth conditions on $k$ at the same time. To this end, we choose $$M=  \lfloor \sqrt{\pi(x)} \log \log x \rfloor$$ and let $k = k(x)$ run over positive even integers such that $\frac{\log k}{\sqrt{x}\log x} \to \infty$ as $x \to \infty.$ 
\end{remark}
\bigskip

\item[{\bf(2)}] For the above choice of $M = M(x),$ we then derive, for every $n \geq 1,$ 
the limit of the moments
$$\left\langle \left(\frac{S^{\pm}(M,f)(x)}{\sqrt{\pi(x)(\mu_{\infty}(I) - \mu_{\infty}(I)^2))}}\right)^n\right\rangle$$
as $x \to \infty.$  In the next section, we show (see Theorem \ref{MomentsAll}) that these converge to the Gaussian moments under the  growth conditions on weight $k$ imposed in the previous step.   \bigskip

\item[{\bf(3)}] Convergence in mean square implies convergence in distribution (see, for example, \cite[Chapter 6, Theorems 5 and 7]{RS}.  Thus, steps \textbf{(1)} and \textbf{(2)} give us 
$$\lim_{x \to \infty}\left\langle \left(\frac{N_I(f,x) - \pi(x)\mu_{\infty}(I)}{\sqrt{\pi(x) (\mu_{\infty}(I) - \mu_{\infty}(I)^2))}}\right)^n\right\rangle$$
for every $n \geq 1.$  These match the moments of the Gaussian distribution. Since the Gaussian distribution is characterized by its moments, one deduces Theorem \ref{main-eigenvalues}.
  
\end{itemize}

\bigskip
Towards the first step, we prove the following proposition:
\begin{prop}\label{mean-squares}
Let $I = [a,b] \subset [-2,2]$ be a fixed interval.  Let $M = \lfloor \sqrt{\pi(x)} \log \log x \rfloor.$  Suppose $k = k(x)$ runs over positive even integers such that $\frac{\log k}{\sqrt{x}\log x} \to \infty$ as $x \to \infty.$  Then,
$$\lim_{x \to \infty}\left\langle \left|\frac{N_I(f,x) - \pi(x)\mu_{\infty}(I) - S^{\pm}(M,f)(x)}{\sqrt{\pi(x)(\mu_{\infty}(I) - \mu_{\infty}(I)^2)}}\right|^2\right\rangle = 0.$$

\end{prop}
\begin{proof}
From equation (\ref{SandNT}), we deduce the following two equations:
\begin{equation}\label{SandNT1}
C\left(\frac{\pi(x)}{M+1}\right) \leq  N_I(f,x) - \pi(x)\mu_{\infty}(I) - S^{-}{(M,f)}(x)
 \leq S^{+}{(M,f)}(x) - S^{-}{(M,f)}(x)+ D\left(\frac{\pi(x)}{M+1}\right)
\end{equation}
and
\begin{equation}\label{SandNT2}
S^{-}{(M,f)}(x) - S^{+}{(M,f)}(x)+ C\left(\frac{\pi(x)}{M+1}\right)
\leq  N_I(f,x) - \pi(x)\mu_{\infty}(I) - S^{+}{(M,f)}(x) \leq  D\left(\frac{\pi(x)}{M+1}\right).
\end{equation}
Thus, for $M \geq 1$ and a suitable positive constant $E,$
\begin{equation*}
\begin{split}
&\langle (N_I(f,x) - \pi(x)\mu_{\infty}(I) - S^{\pm}{(M,f)}(x))^2 \rangle\\
&\leq \max \left\{ \left(\frac{E\pi(x)}{M+1}\right)^2,\,\left\langle \left(S^{+}{(M,f)}(x) - S^{-}{(M,f)}(x) + \frac{E\pi(x)}{M+1}\right)^2\right\rangle \right\}\\
& \leq  \left(\frac{E\pi(x)}{M+1}\right)^2 + \max\left\{0,\,\langle (S^{+}{(M,f)}(x) - S^{-}{(M,f)}(x))^2\rangle + 2\left(\frac{E\pi(x)}{M+1}\right)\langle S^{+}{(M,f)}(x) - S^{-}{(M,f)}(x)\rangle \right\}\\
\end{split}
\end{equation*}
We observe,
$$\langle S^{+}{(M,f)}(x) - S^{-}{(M,f)}(x)\rangle = \O\left(\log \log x + \frac{\pi(x) x^{Mc}}{k}\right).$$
Moreover, combining equations (\ref{SquaresVM}) and (\ref{plusminusVM}), we know that for any $M \geq 3,$
\begin{equation*}
\begin{split}
&\langle (S^{+}{(M,f)}(x) - S^{-}{(M,f)}(x))^2\rangle\\
&= \langle S^{+}{(M,f)}(x)^2 + S^{-}{(M,f)}(x)^2 - 2S^{+}{(M,f)}(x)S^{-}{(M,f)}(x)\rangle\\
& = \O\left(\frac{\pi(x)}{M+1} + (\log \log x)^2 + \frac{\pi(x)^2 x^{2Mc}}{k}\right).
\end{split}
\end{equation*}
From the above, we deduce
\begin{equation*}
\begin{split}
&\langle (N_I(f,x) - \pi(x)\mu_{\infty}(I) - S^{\pm}{(M,f)}(x))^2 \rangle\\
&\ll \frac{\pi(x)^2}{(M+1)^2} + \frac{\pi(x)}{M+1}  + (\log \log x)^2 + \frac{\pi(x)^2 x^{2Mc}}{k} + \frac{\pi(x)}{(M+1)}\left(\log \log x + \frac{\pi(x) x^{Mc}}{k}\right)
\end{split}
\end{equation*}
Thus,
\begin{equation}\label{Msquare}
\begin{split}
&\left\langle \left|\frac{N_I(f,x) - \pi(x)\mu_{\infty}(I) - S^{\pm}(M,f)(x)}{\sqrt{\pi(x)(\mu_{\infty}(I) - \mu_{\infty}(I)^2)}}\right|^2\right\rangle\\
&\ll \frac{\pi(x)}{(M+1)^2} + \frac{1}{M+1} + \frac{(\log \log x)^2}{\pi(x)} + \frac{\pi(x) x^{2Mc}}{k} + \frac{1}{(M+1)}\left(\log \log x + \frac{\pi(x) x^{Mc}}{k}\right)
\end{split}
\end{equation}
We now choose
$$M = \lfloor \sqrt{\pi(x)}\log \log x \rfloor.$$
Thus,
$$\lim_{x \to \infty}\frac{\pi(x)}{(M+1)^2} = 0.$$
Suppose $k = k(x)$ runs over positive even integers such that $\frac{\log k}{\sqrt{x}\log x} \to \infty$ as $x \to \infty.$

\bigskip

Let us fix $0 < d < 1.$  The above growth condition on $k$ tells us that for sufficiently large values of $x,$
$$2c\sqrt{\pi(x)}\log \log x \log x + 1 < d\log k.$$
Thus,
$$\pi(x) x^{2cM} \ll k^d$$
and 
$$\lim_{x \to \infty} \frac{\pi(x) x^{2cM}}{k} = 0.$$

This proves the proposition.
\end{proof}




\section{Higher moments} \label{odd-m}
Henceforth, we set 
$$ T_M^{\pm}(x) := \frac{S^{\pm}(M,f)(x)}{\sqrt{\pi(x)}}$$
and evaluate the moments 
$$\frac{1}{s_k}\sum_{f \in \mathcal F_k}\left(T_M^{\pm}(x)\right)^n$$
for positive integers $n \geq 3$ with $M = \lfloor \sqrt{\pi(x)}\log \log x \rfloor$. 
\begin{remark}
The task of this section is to ascertain how the $n$-th moment of $T_M^{\pm}(x)$ depends on $M$ and prove that the moments indeed converge to the desired limit as $x\to \infty$ for this choice of $M$.
\end{remark}
 By definition, we have
 $$\left(T^{\pm}_M(x)\right)^{n} = \dfrac{1}{{\pi(x)}^{\frac{n}{2}}}\left[ \sum\limits_{m=1}^{M-2}(\A(m)-\A(m+2))\sum\limits_{p\leq x}a_f(p^m) + \A(M-1)\sum\limits_{p\leq x}a_f(p^{M-1}) + \A(M)\sum\limits_{p\leq x}a_f(p^M)\right]^{n}.$$
  For a prime $p,$ we have,
  \begin{equation*}
  Y^{\pm}_M(p) =  \sum\limits_{m=1}^{M}\U(m)a_f(p^m),
  \end{equation*}
  where, as before, we denote, for $M\geq 3$ and $1 \leq m \leq M,$
$$\U(m) :=
\begin{cases}
\A(m)-\A(m+2),&\text{ if }1 \leq m \leq M-2\\
\A(m),&\text{ if }m = M-1,\,M.
\end{cases}
$$
  
  Therefore, 
  $$\left(T^{\pm}_M(x)\right)^{n} = \dfrac{1}{{\pi(x)}^{\frac{n}{2}}}\left(\sum\limits_{p\leq x} Y^{\pm}_M(p)\right)^{n}.$$
 Using the multinomial formula, we may write the above equation as follows.
 \begin{equation} \label{oddm}
\left(T^{\pm}_M(x)\right)^{n} = \dfrac{1}{{\pi(x)}^{\frac{n}{2}}}\sum\limits_{u=1}^{n} \sum_{(r_1, r_2, \ldots,r_u)}^{(1)} \dfrac{n!}{r_1!r_2! \ldots r_u!} \dfrac{1}{u!} \sum\limits_{(p_1,p_2,\ldots,p_u)}^{(2)}Y^{\pm}_M(p_1)^{r_1}Y^{\pm}_M(p_2)^{r_2}\ldots Y^{\pm}_M(p_u)^{r_u},
\end{equation}
where,
\begin{enumerate}[label=(\alph{*})]
 \item The sum $ \sum\limits_{(r_1, r_2, \ldots,r_u)}^{(1)}$ is taken over tuples of positive integers $r_1, r_2, \ldots,r_u$ so that\\ $r_1+ r_2+ \cdots + r_u = n,$ that is, a partition of $n$ into $u$ positive parts.\\
 \item The sum $\sum\limits_{(p_1,p_2,\ldots,p_u)}^{(2)}$ is over $u$-tuples of distinct primes not exceeding $x$.
 \end{enumerate}
We first focus on the inner sum in equation (\ref{oddm}),
\begin{equation}\label{oddm-inner}
 \sum\limits_{(p_1,p_2,\ldots,p_u)}^{(2)}Y^{\pm}_M(p_1)^{r_1}Y^{\pm}_M(p_2)^{r_2}\ldots Y^{\pm}_M(p_u)^{r_u}
\end{equation}
for a fixed partition $(r_1, r_2, \ldots,r_u)$ of $n.$ 
\\
By repeated use of Lemma \ref{Hecke-multiplicative}, we may write, for each $1 \leq i \leq u,$
 \begin{equation}\label{Y_M}
 \begin{split}
 Y^{\pm}_M(p_i)^{r_i} &= \sum\limits_{\underline{m}_i}^{(3)} \U(\underline{m}_i)\left[ D_{r_i,\underline{m}_i}(0)+\sum\limits_{t \in \mathcal{I}(\underline{m}_i) \atop {t \geq 1}}D_{r_i,\underline{m}_i}(t)a_f(p_i^{t})\right]\\
 &=C^{\pm}_M(i) + \sum\limits_{\underline{m}_i}^{(3)} \U(\underline{m}_i)\sum\limits_{t\in \mathcal{I}(\underline{m}_i) \atop {t \geq 1}}D_{r_i,\underline{m}_i}(t)a_f(p_i^{t}),
 \end{split}
 \end{equation} 
 where\\
 
 1. $\underline{m}_i$ denotes an $r_i$-tuple $(m_{j_1},\ldots,m_{j_{r_i}})$.\\
 \bigskip
 
 2. $\sum\limits_{\underline{m}_i}^{(3)}$ denotes that the sum is taken over $r_i$-tuples $\underline{m}_i$ where $1\leq m_{j_l}\leq M$ for each $1 \leq l \leq r_i .$\\
 \bigskip
 
 3. The term $\hat{\mathcal{U}}^{\pm}(\underline{m}_i)$ denotes the product $\hat{\mathcal{U}}^{\pm}(m_{j_1})\ldots\hat{\mathcal{U}}^{\pm}(m_{j_{r_i}})$.\\
 \bigskip
 
 4. For each $r_i$-tuple $\underline{m}_i$, $\mathcal{I}(\underline{m}_i)$ denotes the set of non-negative integers $t$ that occur in the power of $p_i$ on using Lemma \ref{Hecke-multiplicative} and for each $t \in \mathcal{I}(\underline{m}_i),\,D_{r_i,\underline{m}_i}(t)$ denotes the coefficient of $a_f(p_i^{t})$ so obtained. Note that $\mathcal{I}(\underline{m}_i)$ is a finite set for each $\underline{m}_i$ depending on the parity of the sum $m_1 + m_2 + \cdots +m_{r_i}$.  In fact, using the Hecke-multiplicative relations, one deduces the following:
 \begin{equation} \label{I(m)}
 \mathcal{I}(m_1, \ldots, m_{r_i}) \subseteq
 \begin{cases}
\{0, 2, \ldots, m_1+ \cdots+m_{r_i}\}, & \text{if } m_1 +\cdots + m_{r_i} \text{ is even }\\
\{1, 3, \ldots, m_1 + \cdots + m_{r_i} \}&\text{otherwise.}
 \end{cases}
 \end{equation}  We observe that $D_{r_i,\underline{m}_i}(t)$ is independent of the prime $p_i.$  
 \bigskip
 
5. $C^{\pm}_M(i)$ is the sum of the coefficients of $a_f(1)=1$, coming from the the expansion using Lemma \ref{Hecke-multiplicative}. That is,
$$C^{\pm}_M(i) = \sum\limits_{\underline{m}_i}^{(3)} \U(\underline{m}_i)D_{r_i,\underline{m}_i}(0).$$
Observe that $C^{\pm}_M(i)$ is independent of the prime $p_i$ and is in fact a polynomial expression in $\hat{\mathcal{S}}^{\pm}(m)$, $1\leq m\leq M$.\\
 \bigskip
We now prove the following proposition:

\begin{prop}\label{Dt}
Let $1 \leq i \leq u$ and $\underline{m}_i$ be an $r_i$-tuple as specified above.  Then, for $t \in \mathcal{I}(\underline{m}_i),$
$$D_{r_i,\underline{m}_i}(t) = 
\begin{cases}
0,&\text{ if }r_i = 1,\,t = 0\\
1,&\text{ if }r_i = 1,\,t \geq 1\\
\O(1),&\text{ if }r_i = 2,\,t \geq 0\\
\O\left(M^{r_i - 2}\right),&\text{ if }r_i \geq 3,\,t\geq 1\\
\O\left(M^{r_i - 3}\right),&\text{ if }r_i \geq 3,\,t = 0.\\
\end{cases}$$
\end{prop}

\begin{proof}
While focusing on an $r_i$-tuple $\underline{m}_i,$ we may also denote $D_{r_i,\underline{m}_i}(t)$ as $D_{r_i}(t)$ for brevity.

The cases $r_i = 1,\,2$ are clear.  In fact, for $r_i = 2,$
we have
$$Y_M^{\pm}(p)^2 = \sum_{m_1,\,m_2 = 1}^M \U(m_1)\U(m_2)\sum_{i = 0}^{\min\{m_1,\,m_2\}}a_f(p^{m_1+m_2 - 2i})$$
$$ = \sum_{m_1,\,m_2 = 1}^M \U(m_1)\U(m_2) \sum_{t \in \mathcal{I}(m_1,m_2)}a_f(p^{t}),$$
so that the coefficient of $a_f(p^t) =1$ if $t\in \mathcal{I}(m_1,m_2)$ and zero otherwise.  In particular, if $t=0$,
\begin{equation}\label{D_2(0)}
 D_{2,(m_1,m_2)}(0) = 
 \begin{cases}
 1 & \text{if }m_1=m_2\\
 0 &\text{otherwise.}
 \end{cases}
\end{equation}

Using equation (\ref{I(m)}) for $r_i=2$, 
$$|\mathcal{I}(m_1,m_2)| \leq \left(\frac{m_1+m_2}{2}\right) + 1 \leq M+1.$$ 
We now address the case $r_i = 3.$  
Let $l \in \mathcal{I}(m_1,m_2,m_3).$
The product
$$a_f(p^{m_1})a_f(p^{m_2})a_f(p^{m_3}) $$
equals
$$ a_f(p^{m_3})\sum_{i = 0}^{\min\{m_1,\,m_2\}}a_f(p^{m_1+m_2 - 2i}).$$
We observe that in the above product,
$a_f(p^l)$ can occur at most in all possible expansions 
$$a_f(p^{m_3})a_f(p^j),\,j \in \mathcal{I}(m_1,m_2).$$
Since $D_2(t) = 1$ for all $t \in \mathcal{I}(m_1,m_2)$ and $|\mathcal{I}(m_1,m_2) | \leq M+1,$ we deduce
$$D_3(l) \leq M+1.$$
This proves $D_3(r_i) = \O(M^{r_i-2})$ for $r_i = 3.$ 

We now proceed by induction.  Assume that for some $k \geq 3,$ $D_k(l) = \O(M^{k-2}).$
We observe that for each $k$-tuple $\underline{m}_i$, 
\begin{equation}\label{k-exp}
\begin{split}
|\mathcal{I}(\underline{m}_i) | &\leq \left\lfloor \frac{m_1+m_2 + \cdots + m_k}{2}\right\rfloor + 1\\
&\leq \left\lfloor \frac{kM}{2} \right\rfloor + 1
 = \O_k(M).
 \end{split}
 \end{equation}
Now, in the expansion $$(a_f(p^{m_1})a_f(p^{m_2})\cdots a_f(p^{m_k}))a_f(p^{m_{k+1}})$$
$$= a_f(p^{m_{k+1}})\sum_{t \in \mathcal{I}(m_1,m_2,\dots ,m_k)}D_k(t)a_f(p^t),$$
any $a_f(p^l)$ can occur at most in all possible expansions
$$a_f(p^{m_{k+1}})a_f(p^j),\,j \in \mathcal{I}(m_1,m_2,\dots,m_k).$$
By induction hypothesis,
$$D_k(j) = O_k(M^{k-2}),\,j \in \mathcal{I}(m_1,m_2,\dots,m_k).$$
Thus, by equation (\ref{k-exp}), we have
\begin{equation}\label{ind-exp}
D_{k+1}(l) \leq |\mathcal{I}(\underline{m}_i) | |D_k(l)| = \O_k(M^{k-1}).
\end{equation}
Thus, by induction, we have proved that if $r_i \geq 3,\,t \geq 0,$
$$D_{r_i}(t) = \O\left(M^{r_i - 2}\right).$$
Note that the implied constant depends on $r_i.$
We now use these estimates to get a better estimate for $D_{{r_i}}(0)$ for $r_i\geq 3$. We  prove
$$D_{r_i}(0) = \O\left(M^{r_i - 3}\right),\,r_i \geq 3.$$
Equation (\ref{D_2(0)}) tells us that for $r_i = 2,\,D_{r_i}(0) \leq 1.$

For $r_i = 3,$ looking again at the expansion
$$a_f(p^{m_1})a_f(p^{m_2})a_f(p^{m_3}) = a_f(p^{m_3})\sum_{j \in \mathcal{I}(m_1,m_2)}D_2(j)a_f(p^j)$$
$$=\sum_{j \in \mathcal{I}(m_1,m_2)}\sum_{i = 0}^{\min\{j,m_3\}}D_2(j) a_f(p^{m_3+j - 2i}),$$
we observe that $m_3+j-2i = 0$ if and only if $i=j = m_3.$  Thus, 
$$D_3(0) \leq D_2(m_3)=\O(1).$$\\
In general, for ${r_i}\geq 3$,
$$a_f(p^{m_1})\cdots a_f(p^{m_{{r_i}-1}})a_f(p^{m_{r_i}}) = a_f(p^{m_{r_i}})\sum_{j \in \mathcal{I}(m_1,\ldots,m_{{r_i}-1})}D_{{r_i}-1}(j)a_f(p^j)$$
$$= \sum_{j \in \mathcal{I}(m_1,\ldots,m_{{r_i}-1})}\sum_{i = 0}^{\min\{j,m_{r_i}\}}D_{{r_i}-1}(j) a_f(p^{m_{r_i}+j - 2i}).$$ As before, $m_{r_i} +j-2i =0$ if and only if $i=j=m_{r_i}$. Therefore, $$D_{r_i}(0) \leq D_{{r_i}-1}(m_{r_i}) = \O(M^{{r_i}-1-2}) = \O(M^{r_i-3}).$$ 
Here, the implied constant depends on $r_i.$
This proves the proposition.
\end{proof}
We record the following lemma.
  \begin{lem}\label{r_i=2:V_M}
 For $r_i=2$,  $C^{\pm}_M(i) = \sum\limits_{m=1}^M \U({m})^2$. Furthermore, if we let $M= \lfloor\sqrt{\pi(x)}\log\log x \rfloor$, the following holds.
 $$\lim\limits_{x\to \infty} \sum\limits_{m=1}^M \U({m})^2 = \mu_{\infty}(I) - \mu_{\infty}(I)^2.$$
 \end{lem}
 \begin{proof}
 Observe that for $r_i=2$, from equation (\ref{D_2(0)}), it follows that $$C_M^{\pm}(i) = \sum\limits_{m=1}^M \U({m})^2. $$ For the second assertion, note that
 \begin{equation*}
 \begin{split}
 \dfrac{1}{\pi(x)}\langle (S^{\pm}(M,f)(x))^2\rangle &= \dfrac{1}{\pi(x)s_k}\sum_{f \in \mathcal F_k}\sum_{m_1,m_2=1}^{M}\U(m_1)\U(m_2)\sum_{p_1,\,p_2 \leq x} a_f(p_1^{m_1})a_f(p_2^{m_2})\\
 &= \sum\limits_{m=1}^M \U({m})^2 + \O\left((\log\log x)^2 + \dfrac{\pi(x)^2x^{2Mc}}{k}\right),
 \end{split}
 \end{equation*}
using (\ref{2ndMoment-error}) for the sum over $p_1 \neq p_2$ and a similar calculation for the case $p_1 = p_2$ with $ m_1 \neq m_2$. We now plug in our choice of $M$ and compare the above equation with (\ref{SquaresVM}). The claim follows by uniqueness of limits on letting $x \to \infty$.
 \end{proof}

 Taking the product of $Y^{\pm}_M(p_i)^{r_i}$ over $i=1,\dots,u$, we may write (\ref{oddm-inner}) as
 \begin{equation}\label{Yprod}
  \sum\limits_{(p_1,p_2,\dots,p_u)}^{(2)}Y^{\pm}_M(p_1)^{r_1}\dots Y^{\pm}_M(p_u)^{r_u}=  
  \end{equation}
   $$\sum\limits_{(p_1,p_2,\dots,p_u)}^{(2)}\sum\limits_{(\underline{m}_1,\dots, \underline{m}_u)}\U(\underline{m}_1,\dots, \underline{m}_u)\sum\limits_{(t_1,\dots,t_u)}^{(4)}D_{\underline{r},\underline{m}}(\underline{t})a_f(p_1^{t_1}\dots p_u^{t_u}).
 $$
where \\
 \\
 1. $\sum\limits_{(t_1,\dots,t_u)}^{(4)}$ denotes that the sum is taken over $u$-tuples $\underline{t} = (t_1,\dots,t_u),$ where each $t_i \geq 0,$ unless otherwise specified and $t_i \in \mathcal{I}(\underline{m}_i).$  
 
 2. We abbreviate the notation by setting $$\U(\underline{m}_1,\dots,\underline{m}_u):=\U(\underline{m}_1)\dots\U(\underline{m}_u)$$ 
  and for a given tuple $\underline{m}=(\underline{m}_1,\dots,\underline{m}_u)$, 
  $$D_{\underline{r},\underline{m}}(\underline{t}):=D_{r_1,\underline{m}_1}(t_1)D_{r_2,\underline{m}_2}(t_2)\dots D_{r_u,\underline{m}_u}(t_u).
  $$

We now prove the following proposition:
 \begin{prop} \label{oddmoments-crux}
 Suppose $k = k(x)$ runs over positive even integers such that $\frac{\log k}{\sqrt{x}\log x} \to \infty$ as $x \to \infty.$  Let $M = \lfloor \sqrt{\pi(x)} \log \log x \rfloor.$  For each partition $(r_1, r_2,\ldots,r_u)$ of $n$, $$\lim\limits_{x\rightarrow\infty}\dfrac{1}{{\pi(x)}^{\frac{n}{2}}} \dfrac{1}{s_k}\sum\limits_{f\in \F_k}\sum\limits_{(p_1,p_2,\dots,p_u)}^{(2)}Y^{\pm}_M(p_1)^{r_1}Y^{\pm}_M(p_2)^{r_2}\ldots Y^{\pm}_M(p_u)^{r_u}$$ $$= 
 \begin{cases}
 (\mu_\infty(I)-\mu_\infty(I)^2)^{n/2} & \text{if } (r_1, r_2,\dots,r_u)=(2,\dots,2)\\
 0 &\text{otherwise.}
 \end{cases}$$
 \end{prop}

 \begin{proof}
 From equation (\ref{Yprod}), we have, for each partition $(r_1,\ldots,r_u)$ of $n,$ 
 \begin{equation*}
  \frac{1}{\pi(x)^{n/2}}\frac{1}{s_k}\sum_{f \in \mathcal F_k} \sum\limits_{(p_1,p_2,\ldots,p_u)}^{(2)}Y^{\pm}_M(p_1)^{r_1}\cdots Y^{\pm}_M(p_i)^{r_u}=  
   \end{equation*}
   $$\frac{1}{\pi(x)^{n/2}}\frac{1}{s_k}\sum_{f \in \mathcal F_k}\sum\limits_{(p_1,p_2,\ldots,p_u)}^{(2)}\sum\limits_{(\underline{m}_1,\ldots, \underline{m}_u)}\U(\underline{m}_1,\ldots, \underline{m}_u)\sum\limits_{(t_1,\ldots,t_u)}^{(4)}D_{\underline{r},\underline{m}}(\underline{t})a_f(p_1^{t_1}\cdots p_u^{t_u}).
  $$
 For each tuple $(\underline{m}_1,\ldots,\underline{m}_u)$, on applying Proposition \ref{trace}, we have
 $$\dfrac{1}{{\pi(x)}^{\frac{n}{2}}} \left(\dfrac{1}{s_k}\sum\limits_{f\in \F_k}\sum\limits_{(p_1,p_2,\ldots,p_u)}^{(2)}\sum\limits_{(t_1,\ldots,t_u)}^{(4)}D_{\underline{r},\underline{m}}(\underline{t}) a_f(p_1^{t_1}\cdots p_u^{t_u}) \right)=$$
 $$\dfrac{1}{\pi(x)^{\frac{n}{2}}}\sum\limits_{(p_1,p_2,\ldots,p_u)}^{(2)}\sum\limits_{(t_1,\ldots,t_u)}^{(4)}D_{\underline{r},\underline{m}}(\underline{t})\left(\dfrac{\delta(t_1,\ldots,t_u)}{(p_1^{t_1}\cdots p_u^{t_u})^\frac{1}{2}} + \O\left(\dfrac{(p_1^{t_1}\cdots p_u^{t_u})^c}{k}\right)\right),$$
 where $\delta(t_1,\ldots,t_u)=1$ if $2|t_i$ for every $t_i>0$ and  $\delta(t_1,\ldots,t_u)=0$ otherwise. Observe that for each $1 \leq i \leq u,$ $t_i$ is even if and only if the sum of the components of the corresponding $\underline{m}_i$ is even.  
 \\
 The sum 
 \begin{equation}\label{breakdown}
 \begin{split}
& \dfrac{1}{s_k}\sum\limits_{f\in \F_k}\dfrac{1}{{\pi(x)}^{\frac{n}{2}}}\sum\limits_{(p_1,p_2,\ldots,p_u)}^{(2)}Y^{\pm}_M(p_1)^{r_1}\cdots Y^{\pm}_M(p_i)^{r_u}\\
 &=\dfrac{1}{{\pi(x)}^{\frac{n}{2}}}\sum\limits_{(p_1,p_2,\ldots,p_u)}^{(2)}\sum\limits_{(\underline{m}_1,\ldots, \underline{m}_u)}^{(\star)}\U(\underline{m}_1,\ldots, \underline{m}_u) \sum\limits_{(t_1,\ldots,t_u)}^{(4)}D_{\underline{r},\underline{m}}(\underline{t})\dfrac{1}{(p_1^{t_1}\cdots p_u^{t_u})^\frac{1}{2}}\\
 &\quad + \O\left(\dfrac{1}{{\pi(x)}^{\frac{n}{2}}}\sum\limits_{(p_1,p_2,\ldots,p_u)}^{(2)}\sum\limits_{(\underline{m}_1,\ldots, \underline{m}_u)}|\U(\underline{m}_1,\ldots, \underline{m}_u)| \sum\limits_{(t_1,\ldots,t_u)}^{(4)}D_{\underline{r},\underline{m}}(\underline{t})\dfrac{(p_1^{t_1}\cdots p_u^{t_u})^c}{k}\right),
 \end{split}
 \end{equation}
where\\ $\sum\limits_{(\underline{m}_1,\ldots, \underline{m}_u)}^{(\star)}$ denotes that the sum is over those tuples such that $\delta(t_1,\ldots,t_u)=1$.
\\
The technical part of the proof lies in the analysis of the main term of equation (\ref{breakdown}), which is 
$$\dfrac{1}{{\pi(x)}^{\frac{n}{2}}}\sum\limits_{(p_1,p_2,\ldots,p_u)}^{(2)}\sum\limits_{(\underline{m}_1,\dots, \underline{m}_u)}^{(\star)}\U(\underline{m}_1,\ldots, \underline{m}_u) \sum\limits_{(t_1,\dots,t_u)}^{(4)}D_{\underline{r},\underline{m}}(\underline{t})\dfrac{1}{(p_1^{t_1}\dots p_u^{t_u})^\frac{1}{2}}$$
$$=\dfrac{1}{{\pi(x)}^{\frac{n}{2}}}\sum\limits_{(p_1,p_2,\ldots,p_u)}^{(2)}\left(\sum\limits_{\underline{m}_1}^{(\star)}\U(\underline{m}_1)\sum\limits_{{t_1\geq 0}}^{(4)}\dfrac{D_{r_1,\underline{m}_1}(t_1)}{p_1^{t_1/2}}\right)\cdots \left(\sum\limits_{\underline{m}_u}^{(\star)}\U(\underline{m}_u)\sum\limits_{{t_u\geq 0}}^{(4)}\dfrac{D_{r_u,\underline{m}_u}(t_u)}{p_u^{t_u/2}}\right).$$ The idea is to extract the terms where $t_i=0$ for each $i =1, \ldots, u$ and show that the remaining terms are negligible as $x\to \infty$.  

To this end, we write each 
$$\left(\sum\limits_{\underline{m}_i}^{(\star)}\U(\underline{m}_i)\sum\limits_{{t_i\geq 0}}^{(4)}\dfrac{D_{r_i,\underline{m}_i}(t_i)}{p_i^{t_i/2}}\right)$$
as
$$\sum\limits_{\underline{m}_i}^{(\star)} \U(\underline{m}_i)D_{r_i,\underline{m}_i}(0) + \sum\limits_{\underline{m}_i}^{(\star)}\U(\underline{m}_i)\sum\limits_{{t_i\geq 2}}^{(4)}\dfrac{D_{r_i,\underline{m}_i}(t_i)}{p_i^{t_i/2}}.$$

Therefore, denoting
$$C_M^{\pm}(i) = \sum\limits_{\underline{m}_i}^{(\star)} \U(\underline{m}_i)D_{r_i,\underline{m}_i}(0),$$
we have, for a partition $(r_1,\,r_2,\dots\,r_u)$ of $n,$
\begin{equation}\label{D(t)prod1}
\begin{split}
&\dfrac{1}{{\pi(x)}^{\frac{n}{2}}}\sum\limits_{(p_1,p_2,\ldots,p_u)}^{(2)}\left(\sum\limits_{\underline{m}_1}^{(\star)}\U(\underline{m}_1)\sum\limits_{{t_1\geq 0}}^{(4)}\dfrac{D_{r_1,\underline{m}_1}(t_1)}{p_1^{t_1/2}}\right)\cdots \left(\sum\limits_{\underline{m}_u}^{(\star)}\U(\underline{m}_u)\sum\limits_{{t_u\geq 0}}^{(4)}\dfrac{D_{r_u,\underline{m}_u}(t_u)}{p_u^{t_u/2}}\right)\\
&= \dfrac{1}{{\pi(x)}^{\frac{n}{2}}}\sum\limits_{(p_1,p_2,\ldots,p_u)}^{(2)} \left(\prod\limits_{i=1}^{u}C_M^{\pm}(i)\right)\\
&\quad +\dfrac{1}{{\pi(x)}^{\frac{n}{2}}}\sum\limits_{(p_1,p_2,\ldots,p_u)}^{(2)}\sum\limits_{(\varepsilon_1,\ldots,\varepsilon_u)}\prod\limits_{i=1}^{u} \left(C_M^{\pm}(i)\right)^{1-\varepsilon_i}\left(\sum\limits_{\underline{m}_i}^{(\star)}\U(\underline{m}_i)\sum\limits_{{t_i\geq 2}}^{(4)}\dfrac{D_{r_i,\underline{m}_i}(t_i)}{p_i^{t_i/2}}\right)^{\varepsilon_i}.
\end{split}
\end{equation}
Here, in the second term on the right hand side, $(\varepsilon_1, \varepsilon_2,\dots, \varepsilon_u)$ runs over all $u$-tuples such that for each $i=1,\dots,u$, the corresponding $\varepsilon_i \, \in \{0,1\}$ and at least one $\varepsilon_i$ is non-zero. The tuple $(0,\ldots,0)$ is accounted for by the first term. We also follow the convention that if $C^{\pm}_M(i)=0$, then $\varepsilon_i$ is fixed to be $1$ and $C^{\pm}_M(i)^{1-\varepsilon_i}=1$.
\\

Let $$\tilde{D}_{\underline{m}_i}(r_i) := \max\{D_{r_i,\underline{m}_i}(t_i):\,t_i \in \mathcal{I}_{\underline{m}_i}\}.$$ \\
Then, we have $$\sum\limits_{\underline{m}_i}^{(\star)}\U(\underline{m}_i)\sum\limits_{{t_i\geq 2}}^{(4)}\dfrac{D_{r_i,\underline{m}_i}(t_i)}{p_i^{t_i/2}} \ll \sum\limits_{\underline{m}_i}^{(\star)}\U(\underline{m}_i)\dfrac{\tilde{D}_{\underline{m}_i}(r_i)}{p_i}.$$
From this, we derive,
$$\dfrac{1}{{\pi(x)}^{\frac{n}{2}}}\sum\limits_{(p_1,p_2,\ldots,p_u)}^{(2)}\left(\sum\limits_{\underline{m}_1}^{(\star)}\U(\underline{m}_1)\sum\limits_{{t_1\geq 0}}^{(4)}\dfrac{D_{r_1,\underline{m}_1}(t_1)}{p_1^{t_1/2}}\right)\cdots \left(\sum\limits_{\underline{m}_u}^{(\star)}\U(\underline{m}_u)\sum\limits_{{t_u\geq 0}}^{(4)}\dfrac{D_{r_u,\underline{m}_u}(t_u)}{p_u^{t_u/2}}\right)$$
$$= \dfrac{1}{{\pi(x)}^{\frac{n}{2}}}\sum\limits_{(p_1,p_2,\ldots,p_u)}^{(2)} \left(\prod\limits_{i=1}^{u}C_M^{\pm}(i)\right) $$
\begin{equation}\label{D(t)prod}
+\O\left( \dfrac{1}{{\pi(x)}^{\frac{n}{2}}}\sum\limits_{(p_1,p_2,\ldots,p_u)}^{(2)}\sum\limits_{(\varepsilon_1,\ldots,\varepsilon_u)}\prod\limits_{i=1}^{u} \left|C_M^{\pm}(i)\right|^{1-\varepsilon_i}\left(\sum\limits_{\underline{m}_i}^{(\star)}|\U(\underline{m}_i)|\dfrac{\tilde{D}_{\underline{m}_i}(r_i)}{p_i}\right)^{\varepsilon_i}\right),
\end{equation}

Consider the error term in the above equation. We prove that this term vanishes as $x\to\infty$ for our choice of $M$ by showing that for each tuple $(\varepsilon_1,\ldots,\varepsilon_u)$, 
$$\lim\limits_{x\to\infty}\sum\limits_{(p_1,p_2,\ldots,p_u)}^{(2)}\prod\limits_{i=1}^{u} \left|C_M^{\pm}(i)\right|^{1-\varepsilon_i}\left(\sum\limits_{\underline{m}_i}^{(\star)}|\U(\underline{m}_i)|\dfrac{\tilde{D}_{\underline{m}_i}(r_i)}{p_i}\right)^{\varepsilon_i} =0.$$ 
First, for each tuple $(\varepsilon_1,\ldots,\varepsilon_u)$ observe that we may write $$\sum\limits_{(p_1,p_2,\ldots,p_u)}^{(2)}\prod\limits_{i=1}^{u} \left|C_M^{\pm}(i)\right|^{1-\varepsilon_i}\left(\sum\limits_{\underline{m}_i}^{(\star)}|\U(\underline{m}_i)|\dfrac{\tilde{D}_{\underline{m}_i}(r_i)}{p_i}\right)^{\varepsilon_i} $$ as $$\sum\limits_{(p_1,p_2,\ldots,p_u)}^{(2)}\left(\prod\limits_{i=1 \atop {\varepsilon_i = 0}}^{u} \left|C_M^{\pm}(i)\right|\right)\prod\limits_{i=1 \atop {\varepsilon_i = 1}}^{u}\left(\sum\limits_{\underline{m}_i}^{(\star)}|\U(\underline{m}_i)|\dfrac{\tilde{D}_{\underline{m}_i}(r_i)}{p_i}\right). $$ For $\underline{\varepsilon} = (\varepsilon_1,\ldots,\varepsilon_u),$ we define
$$\alpha(\underline{\varepsilon}):=\alpha(\varepsilon_1,\ldots,\varepsilon_u):= \#\{ 1\leq i\leq u:\, \varepsilon_i=0\}.$$
 We observe that if $r_i = 1,$ then $C_M^{\pm}(i) = 0.$  In general, for $r_i \geq 2,$ we have
$$ |C_M^{\pm}(i)| \ll \sum\limits_{\underline{m}_i}^{(\star)} |\U(\underline{m}_i)||D_{r_i,\underline{m}_i}(0)|.$$
If $r_i = 2,$ then for each $\underline{m}_i,$
$$D_{r_i,\underline{m}_i}(0) = \O(1).$$
Thus, by Lemma \ref{Um-bounds},
$$|C_M^{\pm}(i)| \ll \sum\limits_{\underline{m}_i}^{(\star)} |\U(\underline{m}_i)| \ll (\log M)^{r_i}.$$
On the other hand, if $r_i \geq 3,$ then, by Proposition \ref{Dt}, for each $\underline{m}_i,$
$$D_{r_i,\underline{m}_i}(0) = \O(M^{r_i-3}).$$
Once again, by Lemma \ref{Um-bounds},
\begin{equation}\label{Dtzero}
\begin{split}
 |C_M^{\pm}(i)|
 &\ll \sum\limits_{\underline{m}_i}^{(\star)} |\U(\underline{m}_i)||D_{r_i,\underline{m}_i}(0)|\\
 & \ll 
 \begin{cases}
 (\log M)^{r_i} &\text{ if } r_i = 1,2\\
  M^{r_i-3}(\log M)^{r_i} &\text{ if } r_i \geq 3.
  \end{cases}
  \end{split}
 \end{equation}
 Similarly, by another application of Proposition \ref{Dt} and Lemma \ref{Um-bounds} , we have
 \begin{equation}\label{Dtnonzero}
 \sum\limits_{\underline{m}_i}^{(\star)}|\U(\underline{m}_i)|\dfrac{\tilde{D}_{\underline{m}_i}(r_i)}{p_i} 
 \ll 
 \begin{cases}
\frac{ (\log M)}{p_i} &\text{ if }r_i =1\\
 \frac{M^{r_i-2}(\log M)^{r_i}}{p_i} &\text{ if }r_i \geq 2.
 \end{cases}
 \end{equation}
 The partition $(r_1,\,r_2,\dots ,r_u)$ can be of two types as described below.
 \\
 
 \textbf{Case 1:} The partition $(r_1,\ldots,r_u)$ satisfies the condition $r_i>1$ for $i=1,\ldots,u$. Observe that this means $u\leq \frac{n}{2}$.\\
 In this case, by equations (\ref{Dtzero}) and (\ref{Dtnonzero}), for each tuple $(\varepsilon_1,\ldots,\varepsilon_u),$ we have
 \begin{equation}
 \begin{split}
 &\sum\limits_{(p_1,p_2,\ldots,p_u)}^{(2)}\left(\prod\limits_{i=1 \atop {\varepsilon_i = 0}}^{u} \left|C_M^{\pm}(i)\right|\right)\prod\limits_{i=1 \atop {\varepsilon_i = 1}}^{u}\left(\sum\limits_{\underline{m}_i}^{(\star)}|\U(\underline{m}_i)|\dfrac{\tilde{D}_{\underline{m}_i}(r_i)}{p_i}\right)\\
 &\ll \sum\limits_{(p_1,p_2,\ldots,p_u)}^{(2)}\left(\prod\limits_{i=1 \atop {\varepsilon_i = 0}}^{u} M^{r_i - 2}(\log M)^{r_i}\right)\left(\prod\limits_{i=1 \atop {\varepsilon_i = 1}}^{u}M^{r_i-2}\frac{(\log M)^{r_i}}{p_i}\right)\\
 &\\
 &\ll M^{n - 2\alpha(\underline{\varepsilon}) - 2(u - \alpha(\underline{\varepsilon}))}(\log M)^n\sum\limits_{(p_1,p_2,\ldots,p_u)}^{(2)}\dfrac{1}{\prod\limits_{{i=1}\atop \varepsilon_i=1}^{u}p_i}\\
 &\\
 &\ll M^{n - 2\alpha(\underline{\varepsilon}) - 2(u - \alpha(\underline{\varepsilon}))}(\log M)^n\pi(x)^{\alpha(\underline{\varepsilon})}(\log \log x)^{u - \alpha(\underline{\varepsilon})}\\
 &\\
 &\ll M^{n-2u}\pi(x)^{\alpha(\underline{\varepsilon})}(\log \log x)^{u - \alpha(\underline{\varepsilon})}(\log M)^n.
 \end{split}
 \end{equation}
 We now choose $M = \lfloor \sqrt{\pi(x)} \log \log x \rfloor.$  The above error term is
 $$\ll \pi(x)^{\frac{n}{2} - u}\pi(x)^{u-1}(\log \log x)^{u} (\log x)^n,$$
 since $\alpha(\underline{\varepsilon}) \leq u-1.$
 Thus, for each tuple $(\varepsilon_1,\ldots,\varepsilon_u),$
 $$\lim_{x \to \infty}\frac{1}{\pi(x)^{n/2}}\sum\limits_{(p_1,p_2,\ldots,p_u)}^{(2)}\left(\prod\limits_{i=1 \atop {\varepsilon_i = 0}}^{u} \left|C_M^{\pm}(i)\right|\right)\prod\limits_{i=1 \atop {\varepsilon_i = 1}}^{u}\left(\sum\limits_{\underline{m}_i}^{(\star)}|\U(\underline{m}_i)|\dfrac{\tilde{D}_{\underline{m}_i}(r_i)}{p_i}\right)$$
 $$ \ll \lim_{x \to \infty}\frac{1}{\pi(x)^{\frac{n}{2}}}\pi(x)^{\frac{n}{2} - 1}(\log \log x)^{u} (\log x)^n = 0.$$
 \\
 \textbf{Case 2:} The partition $(r_1,\ldots,r_u)$  has at least one component $r_i$ equal to 1. Let $l$ be the number of $1$'s in the partition. Without loss of generality, we may assume that the last $l$ parts are equal to one while $r_1,\ldots,r_{u-l}$ are at least $2$. By our convention, since $C^{\pm}_M(i)=0$ if $r_i=1$, we have $\varepsilon_{i}=1$ for $u-l+1\leq i\leq u$. Also, if $r_i=1$, $\tilde{D}_{\underline{m}_i}(r_i)=1$.
 For $\underline{\varepsilon}= (\varepsilon_1,\ldots,\varepsilon_u)=(\varepsilon_1,\ldots,\varepsilon_{u-l},1,\ldots,1)$, let $$\alpha_l(\underline{\varepsilon})= \#\{1\leq i\leq u-l: \varepsilon_i=0\}.$$
 Therefore, if the partition in consideration has $l$ components equal to $1$, we have
 $$\sum\limits_{(p_1,p_2,\ldots,p_u)}^{(2)}\sum\limits_{(\varepsilon_1,\ldots,\varepsilon_u)}\prod\limits_{i=1}^{u} \left|C_M^{\pm}(i)\right|^{1-\varepsilon_i}\left(\sum\limits_{\underline{m}_i}^{(\star)}|\U(\underline{m}_i)|\dfrac{\tilde{D}_{\underline{m}_i}(r_i)}{p_i}\right)^{\varepsilon_i}$$
 $$= \sum\limits_{(p_1,p_2,\ldots,p_u)}^{(2)}\sum\limits_{(\varepsilon_1,\ldots,\varepsilon_{u-l})}\prod\limits_{i=1}^{{u-l}}\left[ \left|C_M^{\pm}(i)\right|^{1-\varepsilon_i}\left(\sum\limits_{\underline{m}_i}^{(\star)}|\U(\underline{m}_i)|\dfrac{\tilde{D}_{\underline{m}_i}(r_i)}{p_i}\right)^{\varepsilon_i}\right]\prod\limits_{i={u-l+1}}^{u}\left( \sum\limits_{\underline{m}_i}^{(\star)}|\U(\underline{m}_i)|\dfrac{1}{p_i}\right).$$
 Again, using equations (\ref{Dtzero}) and (\ref{Dtnonzero}) as well as Lemma \ref{Um-bounds}, for each tuple $(\varepsilon_1,\ldots,\varepsilon_{u-l},1,\ldots,1)$ we have
 \begin{equation}
 \begin{split}
&\sum\limits_{(p_1,p_2,\ldots,p_u)}^{(2)}\left(\prod\limits_{i=1 \atop {\varepsilon_i = 0}}^{u-l} \left|C_M^{\pm}(i)\right|\right)\prod\limits_{i=1 \atop {\varepsilon_i = 1}}^{u-l}\left(\sum\limits_{\underline{m}_i}^{(\star)}|\U(\underline{m}_i)|\dfrac{\tilde{D}_{\underline{m}_i}(r_i)}{p_i}\right)\prod\limits_{i=u-l+1}^{u}\left( \sum\limits_{\underline{m}_i}^{(\star)}|\U(\underline{m}_i)|\dfrac{\tilde{D}_{\underline{m}_i}(1)}{p_i}\right)\\
 &\ll \sum\limits_{(p_1,p_2,\ldots,p_u)}^{(2)}\left(\prod\limits_{i=1 \atop {\varepsilon_i = 0}}^{u-l} M^{r_i - 2}(\log M)^{r_i}\right)\left(\prod\limits_{i=1 \atop {\varepsilon_i = 1}}^{u-l}M^{r_i-2}\frac{(\log M)^{r_i}}{p_i}\right)\dfrac{(\log M)^l}{p_{u-l+1}\cdots p_u}\\
 &\\
 &\ll M^{n - l-2\alpha_l(\underline{\varepsilon}) - 2(u - l-\alpha_l(\underline{\varepsilon}))}(\log M)^n\sum\limits_{(p_1,p_2,\ldots,p_u)}^{(2)}\dfrac{1}{\prod\limits_{{i=1}\atop \varepsilon_i=1}^{u}p_i}\\
 &\\
 &\ll M^{n - 2\alpha_l(\underline{\varepsilon}) - 2(u -l- \alpha_l(\underline{\varepsilon}))}(\log M)^n\pi(x)^{\alpha(\underline{\varepsilon})}(\log \log x)^{u - \alpha_l(\underline{\varepsilon})}\\
 &\\
 &\ll M^{n-l-2(u-l)}\pi(x)^{\alpha_l(\underline{\varepsilon})}(\log \log x)^{u - \alpha_l(\underline{\varepsilon})}(\log M)^n.
 \end{split}
 \end{equation}
 Substituting our chosen value for $M$ and using the bound $\alpha_l(\underline{\varepsilon})\leq u-l$, the above error term is 
 $$\ll \pi(x)^{\frac{n}{2}-\frac{l}{2}}(\log\log x)^{u}(\log x)^n.$$
 Therefore,
 $$\lim\limits_{x\to\infty} \dfrac{1}{\pi(x)^{\frac{n}{2}}} \sum\limits_{(p_1,p_2,\ldots,p_u)}^{(2)}\left(\prod\limits_{i=1 \atop {\varepsilon_i = 0}}^{u-l} \left|C_M^{\pm}(i)\right|\right)\prod\limits_{i=1 \atop {\varepsilon_i = 1}}^{u-l}\left(\sum\limits_{\underline{m}_i}^{(\star)}|\U(\underline{m}_i)|\dfrac{\tilde{D}_{\underline{m}_i}(r_i)}{p_i}\right)\prod\limits_{i=u-l+1}^{u}\left( \sum\limits_{\underline{m}_i}^{(\star)}|\U(\underline{m}_i)|\dfrac{\tilde{D}_{\underline{m}_i}(1)}{p_i}\right)$$
 $$\ll \lim\limits_{x\to\infty} \dfrac{1}{\pi(x)^{\frac{n}{2}}}\pi(x)^{\frac{n}{2}-\frac{1}{2}}(\log\log x)^{u}(\log x)^n=0$$
 noting that $l\geq 1$.

From the analysis in Cases 1 and 2, we deduce that for any partition $(r_1,\,r_2,\dots r_u)$ of $n$, the error term in equation (\ref{D(t)prod}) vanishes in the limit. That is,

 \begin{equation}\label{rhs-2}
\lim_{x \to \infty} \dfrac{1}{{\pi(x)}^{\frac{n}{2}}}\sum\limits_{(p_1,p_2,\ldots,p_u)}^{(2)}\sum\limits_{(\varepsilon_1,\ldots,\varepsilon_u)}\prod\limits_{i=1}^{u} \left(C_M^{\pm}(i)\right)^{1-\varepsilon_i}\left(\sum\limits_{\underline{m}_i}^{(\star)}\U(\underline{m}_i)\sum\limits_{{t_i\geq 0}}^{(4)}\dfrac{D_{r_i,\underline{m}_i}(t_i)}{p_i^{t_i/2}}\right)^{\varepsilon_i} = 0,
\end{equation}
where we are summing over all tuples $(\varepsilon_1, \varepsilon_2,\dots \varepsilon_u)$ with at least one $\varepsilon_i$ is non-zero.

From equations (\ref{D(t)prod1}) and (\ref{rhs-2}), we deduce that for a partition $(r_1,\,r_2,\dots r_u)$ of $n,$ \begin{equation}\label{D(t)prod1a}
\begin{split}
&\lim_{x \to \infty} \dfrac{1}{{\pi(x)}^{\frac{n}{2}}}\sum\limits_{(p_1,p_2,\ldots,p_u)}^{(2)}\left(\sum\limits_{\underline{m}_1}^{(\star)}\U(\underline{m}_1)\sum\limits_{{t_1\geq 0}}^{(4)}\dfrac{D_{r_1,\underline{m}_1}(t_1)}{p_1^{t_1/2}}\right)\cdots \left(\sum\limits_{\underline{m}_u}^{(\star)}\U(\underline{m}_u)\sum\limits_{{t_u\geq 0}}^{(4)}\dfrac{D_{r_u,\underline{m}_u}(t_u)}{p_u^{t_u/2}}\right)\\
&= \lim\limits_{x \to \infty} \dfrac{1}{{\pi(x)}^{\frac{n}{2}}}\sum\limits_{(p_1,p_2,\ldots,p_u)}^{(2)} \left(\prod\limits_{i=1}^{u}C_M^{\pm}(i)\right).
\end{split}
\end{equation}

We now study the term
$$\dfrac{1}{{\pi(x)}^{\frac{n}{2}}}\sum\limits_{(p_1,p_2,\ldots,p_u)}^{(2)} \left(\prod\limits_{i=1}^{u}C_M^{\pm}(i)\right)$$
as $x \to \infty.$
\\
The partitions $(r_1,\,r_2,\dots r_u)$ are of three different types as described below.
\\
\textbf{Case 1:}
If $(r_1,\,r_2,\dots r_u) = (2,\,2,\dots 2),$ then $u = n/2$ and
$$\dfrac{1}{{\pi(x)}^{\frac{n}{2}}}\sum\limits_{(p_1,p_2,\ldots,p_u)}^{(2)} \left(\prod\limits_{i=1}^{u}C_M^{\pm}(i)\right) = \left(\prod\limits_{i=1}^{u}C_M^{\pm}(i)\right)\dfrac{\pi(x)(\pi(x) - 1)(\pi(x) - 2) \dots (\pi(x) - n/2 + 1)}{{\pi(x)}^{\frac{n}{2}}}$$
$$= \dfrac{1}{\pi(x)^{\frac{n}{2}}}\left(\sum_{m_1}^{M}\U(m)^2\right)^{n/2} \left(\pi(x)^{\frac{n}{2}} + \o (\pi(x)^{\frac{n}{2}} )\right),$$
using Lemma {\ref{r_i=2:V_M}}.
\\
\textbf{Case 2:}
 If $r_i = 1$ for some $r_i$ in the given partition, then the corresponding $C_M^{\pm}(i)$ is 0.  Thus,
 $$\dfrac{1}{{\pi(x)}^{\frac{n}{2}}}\sum\limits_{(p_1,p_2,\ldots,p_u)}^{(2)} \left(\prod\limits_{i=1}^{u}C_M^{\pm}(i)\right) = 0.$$
 \\
\textbf{Case 3:} Each $r_i \geq 2$ with at least one $r_i\geq 3.$  Without loss of generality, for some $1 \leq l \leq  u,$ suppose we have $r_1,\,r_2,\,\dots\,r_l \geq 3$ and $r_{l+1} = \dots = r_u = 2.$
\\

Thus, $(r_1 + \dots +r_l) + 2(u-l) = n.$  
 By equation (\ref{Dtzero}),
 $$ \left|\prod\limits_{i=1}^{u}C_M^{\pm}(i)\right| \ll M^{\sum_{i=1}^l (r_i - 3)}(\log M)^n$$
$$ \ll M^{n - 2(u-l)- 3l} (\log M)^n = M^{n - l - 2u}(\log M)^n.$$
Choosing $M = \lfloor \sqrt{\pi(x)} \log \log x \rfloor,$
$$ \frac{1}{\pi(x)^{n/2}}\sum\limits_{(p_1,p_2,\ldots,p_u)}^{(2)}\left|\prod\limits_{i=1}^{u}C_M^{\pm}(i)\right|$$
$$ \ll \frac{1}{\pi(x)^{n/2}} (\pi(x))^{\frac{n}{2} - \frac{l}{2} - u+ u}(\log x)^n.$$
Since $l \geq 1,$
$$\lim_{x \to \infty}\dfrac{1}{{\pi(x)}^{\frac{n}{2}}}\sum\limits_{(p_1,p_2,\ldots,p_u)}^{(2)} \left(\prod\limits_{i=1}^{u}C_M^{\pm}(i)\right) = 0.$$
From the above three cases and the second assertion in Lemma {\ref{r_i=2:V_M}}, we deduce that for $M = \lfloor \sqrt{\pi(x)} \log \log x \rfloor,$
\begin{equation}\label{D(t)prod1b}
\begin{split}
&\lim_{x \to \infty}\dfrac{1}{{\pi(x)}^{\frac{n}{2}}}\sum\limits_{(p_1,p_2,\ldots,p_u)}^{(2)} \left(\prod\limits_{i=1}^{u}C_M^{\pm}(i)\right)\\
\\
&=
\begin{cases}
 (\mu_{\infty}(I)-\mu_{\infty}(I)^2)^{n/2} & \text{if } (r_1, r_2,\ldots,r_u)=(2,\ldots,2)\\
 0 &\text{otherwise.}
 \end{cases}
 \end{split}
\end{equation}
This concludes the analysis of the main term in equation (\ref{breakdown}). We now look at the error term of the same equation, which is
$$\O\left(\dfrac{1}{{\pi(x)}^{\frac{n}{2}}}\sum\limits_{(p_1,p_2,\ldots,p_u)}^{(2)}\sum\limits_{(\underline{m}_1,\ldots, \underline{m}_u)}|\U(\underline{m}_1,\ldots, \underline{m}_u)| \sum\limits_{(t_1,\ldots,t_u)}^{(4)}D_{\underline{r},\underline{m}}(\underline{t})\dfrac{(p_1^{t_1}\cdots p_u^{t_u})^c}{k}\right).$$
We observe that for each $i,$
$$\sum_{t_i \geq 0 \atop {t_i \in \mathcal{I}(\underline{m}_i)}}(p_i^c)^{t_i} \ll p_i^{cr_iM}.$$
Thus, by Proposition \ref{Dt} and Lemma \ref{Um-bounds}, the above error term from equation (\ref{breakdown}) becomes
\begin{equation*}
\begin{split}
&= \O\left(\dfrac{1}{{\pi(x)}^{\frac{n}{2}}}\pi(x)^{u}(\log M)^nM^{n - 2u}\sum\limits_{(p_1,p_2,\ldots,p_u)}^{(2)}\frac{p_1^{cr_1M}p_2^{cr_2M} \dots p_u^{cr_uM}}{k}\right)\\
&= \O\left(\dfrac{1}{{\pi(x)}^{\frac{n}{2}}}\pi(x)^{u}(\log M)^nM^{n - 2u}\pi(x)^u\frac{x^{cnM}}{k}\right).
\end{split}
\end{equation*}
For $M = \lfloor \sqrt{\pi(x)} \log \log x \rfloor,$ this is
\begin{equation}
\O\left(\dfrac{1}{{\pi(x)}^{\frac{n}{2}}}\pi(x)^{u}(\log x)^n(\pi(x))^{n/2 - u}(\log \log x)^n\pi(x)^u\frac{x^{cn\sqrt{\pi(x)}\log \log x}}{k}\right).
\end{equation}
Let
$$\frac{\log k}{\sqrt{x}\log x} \to \infty \text{ as } x \to \infty.$$
Then, for $0 < d <1,$
$$(cn\sqrt{\pi(x)}\log\log x + u) \log x \leq d \log k$$
for sufficiently large values of $x.$ 

In particular, given $n \geq 1$ and $M = \lfloor \sqrt{\pi(x)} \log \log x \rfloor,$
$$\lim\limits_{x\to\infty}\dfrac{1}{{\pi(x)}^{\frac{n}{2}}}\pi(x)^{u}(\log M)^nM^{n - 2u}\pi(x)^u\frac{x^{cnM}}{k} = 0.$$
This proves Proposition \ref{oddmoments-crux}.

 \end{proof}
 
Using Proposition \ref{oddmoments-crux} in equation (\ref{oddm}),
we deduce, under the same assumptions on $M$ and $k$ as above,
\begin{equation}\label{44}
\begin{split}
&\lim_{x \to \infty}\frac{1}{s_k}\sum_{f \in \mathcal F_k}\left(T^{\pm}_M(x)\right)^{n} \\
&= \lim_{x \to \infty} \dfrac{1}{{\pi(x)}^{\frac{n}{2}}}\sum\limits_{u=1}^{n} \sum_{(r_1, r_2, \ldots,r_u)}^{(1)} \dfrac{n!}{r_1!r_2! \ldots r_u!} \dfrac{1}{u!} \sum\limits_{(p_1,p_2,\ldots,p_u)}^{(2)}Y^{\pm}_M(p_1)^{r_1}Y^{\pm}_M(p_2)^{r_2}\ldots Y^{\pm}_M(p_u)^{r_u}\\
&=\sum\limits_{u=1}^{n} \sum_{(r_1, r_2, \ldots,r_u)}^{(1)} \dfrac{n!}{r_1!r_2! \ldots r_u!} \dfrac{1}{u!}
\begin{cases}
(\mu_{\infty}(I)-\mu_{\infty}(I)^2)^{u} & \text{if } (r_1, r_2,\dots,r_u)=(2,\,2,\,2,\dots,2)\\
 0 &\text{otherwise}
 \end{cases}\\
 \\
 &=\begin{cases}
 0 &\text{ if } n \text{ is odd}\\
\dfrac{n!}{(\frac{n}{2})!2^{\frac{n}{2}}}(\mu_{\infty}(I)-\mu_{\infty}(I)^2)^{n/2}&\text{ if }n \text{ is even}
\end{cases}
\end{split}
\end{equation}
Thus, by equation (\ref{44}), we have proved the following theorem:
\begin{thm}\label{MomentsAll}
Let $I = [a,b]$ be a fixed interval in $[-2,2].$  Let $M = \lfloor \sqrt{\pi(x)} \log \log x \rfloor$ and suppose $k = k(x)$ runs over positive even integers such that $\frac{\log k}{\sqrt{x}\log x} \to \infty$ as $x \to \infty.$  Then, for a positive integer $n \geq 1,$
$$\lim_{x \to \infty}\left\langle \left(\frac{\ST}{\sqrt {\pi(x)(\mu_{\infty}(I) - \mu_{\infty}(I)^2) }}\right)^n\right\rangle = 
\begin{cases}  0 &\text{ if }$n$\text{ is odd}\\
\dfrac{n!}{(\frac{n}{2})!2^{\frac{n}{2}}} &\text{ if }n \text{ is even}.
\end{cases}$$
\end{thm}
 Thus, by Proposition \ref{mean-squares} and Theorem \ref{MomentsAll}, the proof of Theorem \ref{main-eigenvalues} follows since convergence in mean square implies convergence in distribution and the Gaussian distribution is characterized by its moments.  

\section*{Acknowledgements}
We are very grateful to Ze\'{e}v Rudnick for valuable inputs and guidance during the preparation of this article.  We would like to thank Amir Akbary, A. Raghuram, Stephan Baier, Baskar Balasubramanyam, Abhishek Banerjee, Anup Biswas and M. Ram Murty for helpful discussions.  We also thank the referees for their suggestions which have helped to improve the presentation of this article.

\end{document}